\documentclass[reqno]{amsart}
\usepackage[T1]{fontenc}
\usepackage{amssymb, amsmath, color, textcomp, url,tikz}
\usetikzlibrary{matrix,arrows}
\usetikzlibrary{fit}
\usetikzlibrary{positioning}
\usepackage[labelformat=empty]{caption}
\usepackage{mdwlist}
\usepackage{etoolbox}

\RequirePackage{ifpdf}
\ifpdf
   \usepackage[pdftex]{hyperref}
\else
   \usepackage[hypertex]{hyperref}
\fi

\usepackage{enumerate,xspace}

\theoremstyle{plain}
\newtheorem{theorem}{Theorem}[section]
\newtheorem{cor}[theorem]{Corollary}
\newtheorem{prop}[theorem]{Proposition}
\newtheorem{lemma}[theorem]{Lemma}

\newtheorem*{claimstern}{Claim}
\newenvironment{claimsternproof}{\noindent\textit{Proof of
    Claim.}}{\hfill\qedsymbol \tiny{ Claim}
\medskip}

\newcounter{claimCount}
\newenvironment{claim}{\medskip \vskip-1mm\noindent
  \refstepcounter{claimCount}\textbf{Claim~\arabic{claimCount}.}}{\medskip}

\newenvironment{claimproof}{\noindent\textit{Proof of
    Claim~\arabic{claimCount}.}}{\hfill\qedsymbol
  \tiny{ Claim~\arabic{claimCount}}
\medskip}

\AtBeginEnvironment{proof}{\setcounter{claimCount}{0}}

\theoremstyle{definition}
\newtheorem{remark}[theorem]{Remark}
\newtheorem{fact}[theorem]{Fact}
\newtheorem{definition}[theorem]{Definition}
\newtheorem{example}[theorem]{Example}
\newtheorem*{question}{Question}
\newtheorem*{notation}{Notation}

\newcommand{\nc}{\newcommand}

\nc{\Z}{\mathbb{Z}}
\nc{\N}{\mathbb{N}}
\nc{\Fp}{\mathbb{F}_p}
\nc{\Proj}{\mathbb{P}}
\nc{\K}{\mathit{K}}
\nc{\E}{\mathit{E}}
\nc\LL{\mathcal L}
\nc\LLD{\mathcal L_{D}}
\nc\Pred{\mathcal P}
\nc\car{\mathrm{char}}

\nc{\acfp}{\mathsf{ACFP}}
\nc{\dcfp}{\mathsf{DCF}_p}
\nc{\scfp}{\mathsf{SCF}_p}
\nc{\scfpe}{\mathsf{SCF}_{p,e}}
\nc{\scfpi}{\mathsf{SCF}_{p,\infty}}

\nc{\Land}{\bigwedge}
\nc{\proves}{\vdash}
\renewcommand{\phi}{\varphi}

\DeclareMathOperator{\dep}{Dep}
\DeclareMathOperator{\pdep}{\text{$p$}\hspace{0.12em}-Dep}
\DeclareMathOperator{\C}{C} 

\DeclareMathOperator{\W}{W}
\DeclareMathOperator{\id}{id}

\nc\ord{\operatorname{ord}}
\nc{\acl}{\operatorname{acl}}
\nc{\ldim}{\operatorname{ldim}}
\nc\inv{ ^{-1}}

\nc{\ldef}[2][]{\overline{\lambda}#1(#2)\!\downarrow}

\nc{\tp}{\operatorname{tp}}
\nc{\cf}{\text{cf. }}
\nc\icl{indiscernibly closed\xspace}

\nc\Ann{\operatorname{Ann}}
\nc{\Gr}{\operatorname{Gr}}
\nc{\pk}{\operatorname{Pk}}
\nc{\pluck}{\operatorname{\lrcorner}}

\makeatletter
\newcommand{\extp}{\@ifnextchar^\@extp{\@extp^{\,}}}
\def\@extp^#1{\mathop{\Land\nolimits^{\!#1}}}
\makeatother

\def\Ind#1#2{#1\setbox0=\hbox{$#1x$}\kern\wd0\hbox to 0pt{\hss$#1\mid$\hss}
\lower.9\ht0\hbox to 0pt{\hss$#1\smile$\hss}\kern\wd0}
\def\Notind#1#2{#1\setbox0=\hbox{$#1x$}\kern\wd0\hbox to
0pt{\mathchardef\nn="0236\hss$#1\nn$\kern1.4\wd0\hss}\hbox
to 0pt{\hss$#1\mid$\hss}\lower.9\ht0
\hbox to 0pt{\hss$#1\smile$\hss}\kern\wd0}
\def\ind{\mathop{\mathpalette\Ind{}}}

\def\ld{\mathop{\ \ \hbox to 0pt{\hss$\mid^{\hbox to
0pt{$\scriptstyle\mathrm{ld}$\hss}}$\hss}
\lower4pt\hbox to 0pt{\hss$\smile$\hss}\ \ }}

\begin{document}

\title{Equational theories of fields}
\date{March 22, 2018}

\author{ Amador Martin-Pizarro and Martin Ziegler}
\address{Mathematisches Institut,
  Albert-Ludwigs-Universit\"at Freiburg, D-79104 Freiburg, Germany}
\email{pizarro@math.uni-freiburg.de}
\email{ziegler@uni-freiburg.de}
\thanks{Research partially supported by the program
  MTM2014-59178-P. Additionally, the first author conducted research
  with support of the program ANR-13-BS01-0006 Valcomo.}

\begin{abstract}
  A complete first-order theory is equational if every definable set
  is a Boolean combination of instances of equations, that is, of
  formulae such that the family of finite intersections of instances
  has the descending chain condition. Equationality is a strengthening
  of stability. We show the equationality of the theory of proper
  extensions of algebraically closed fields of some fixed
  characteristic and of the theory of separably closed fields of
  arbitrary imperfection degree. Srour showed that the theory of
  differentially closed fields in positive characteristic is
  equational. We give also a different proof of his result.
\end{abstract}

\maketitle
\tableofcontents

\section{Introduction}

Consider a sufficiently saturated model of a complete theory $T$. A
formula $\phi(x;y)$ is an \emph{equation} (for a given partition of
the free variables into $x$ and $y$) if the family of finite
intersections of instances $\phi(x,a)$ has the descending chain
condition (DCC). The theory $T$ is \emph{equational} if every formula
$\psi(x;y)$ is equivalent modulo $T$ to a Boolean combination of
equations $\phi(x;y)$.

Quantifier elimination implies that the theory of algebraically closed
fields of some fixed characteristic is equational. Separably closed
fields of positive characteristic have quantifier elimination after
adding \emph{$\lambda$-functions}  to the ring language
\cite{fD88}. The \emph{imperfection degree} of a separably closed
field $K$ of positive characteristic $p$ encodes the linear dimension
of $K$ over $K^p$. If the imperfection degree is finite, restricting
the $\lambda$-functions to a fixed \emph{$p$-basis} yields again
equationality. A similar manipulation yields elimination of
imaginaries for separably closed field $K$ of positive characteristic
and finite imperfection degree, in terms of the field of definition of
the corresponding defining ideal. However, there is not an
explicit description of imaginaries for separably closed fields
$K$ of infinite imperfection degree, that is, when $K$ has infinite
linear dimension over the definable subfield $K^p$.

Another important (expansion of a) theory of fields having infinite
linear dimension over a definable subfield is the theory of an
algebraically closed field with a predicate for a distinguished
algebraically closed proper subfield. Any two such pairs are
elementarily equivalent if and only if they have the same
characteristic. They are exactly the models of the theory of Poizat's
\emph{belles paires} \cite{Po83} of algebraically closed fields.

Determining whether a particular theory is equational is not obvious.
So far, the only known \emph{natural} example of a stable
non-equational theory is the free non-abelian finitely generated group
\cite{zS13, MS17}. In this paper, we will prove the equationality of
several theories of fields: the theory of belles paires of
algebraically closed fields of some fixed characteristic, as well as
the theory of separably closed fields of arbitrary imperfection degree
We also give a new proof of the equationality of the theory of
differentially closed fields in positive characteristic, which was
established by Srour \cite{gS88}. In Section \ref{S:EqPairs} we
include an alternative proof for belles paires of characteristic $0$,
by showing that definable sets are Boolean combination of certain
definable sets, which are Kolchin-closed in the corresponding
expansion $\mathsf{DCF}_0$. A similar approach appeared already in
\cite{aG16} using different methods. We generalise this approach to
arbitrary characteristic in Section \ref{S:LinFor}.

\section{Equations and indiscernible sequences}

Most of the results in this section come from \cite{PiSr84, Ju00,
  mJdL01}. We refer the  avid reader to \cite{aO11} for a gentle
introduction to equationality.

We work inside a sufficiently saturated model $\mathbb U$ of a
complete theory $T$. A formula $\phi(x;y)$, with respect to a given
partition of the free variables into $x$ and $y$, is an
\emph{equation} if the family of finite intersections of instances
$\phi(x,b)$ has the descending chain condition (DCC). If $\phi(x;y)$
is an equation, then so are $\phi\inv(y;x)=\phi(x,y)$ and
$\phi(f(x);y)$, whenever $f$ is a $\emptyset$-definable map. Finite
conjunctions and disjunctions of equations are again equations. By an
abuse of notation, given an incomplete theory, we will say that a
formula is an equation if it is an equation in every completion of the
theory.

\noindent The theory $T$ is \emph{equational} if every formula $\psi(x;y)$ is
equivalent modulo $T$ to a Boolean combination of equations $\phi(x;y)$.

Typical examples of equational theories are the theory of an
equivalence relation with infinite many infinite classes, the theory
of $R$-modules.

\begin{example}\label{E:hilbert}
  In any field $K$, for every polynomial $p(X,Y)$ with integer
  coefficients, the equation $p(x;y)\doteq 0$ is an equation in the
  model-theoretic sense.
\end{example}
\begin{proof}
  This follows immediately from Hilbert's Basis Theorem, which implies
  that the Zariski topology on $K^n$ is noetherian, i.e.\ the system
  of all algebraic sets
  \[\Bigl\{a\in K^n\Bigm|\Land_{i=1}^m q_i(a)=0\Bigr\},\] where
  $q_i\in K[X_1,\dots,X_n]$, has the DCC.

  \noindent There is a simpler proof, without using Hilbert's Basis
  Theorem: Observe first that $p(x;y)\doteq 0$ is an equation, if $p$
  is linear in $x$, since then $p(x;a)\doteq 0$ defines a subspace of
  $K^n$. Now, every polynomial has the form $q(M_1,\ldots,M_m;y)$,
  where $q(u_1,\dotsc,u_m;y)$ is linear in the $u_i$, for some
  monomials $M_1,\ldots, M_m$ in $x$.
\end{proof}

\noindent Quantifier elimination for the incomplete) theory
$\mathsf{ACF}$ of algebraically closed fields and the above example
yield that $\mathsf{ACF}$ is equational.

Equationality is preserved under unnaming parameters and
bi-interpretability \cite{Ju00}. It is unknown whether equationality
holds if every formula $\phi(x;y)$, with $x$ a single variable, is
a boolean combination of equations.

By compactness, a formula $\phi(x;y)$ is an equation if there is no
indiscernible sequence $(a_i,b_i)_{i\in\N}$ such that
$\phi(a_i,b_j)$ holds for $i<j$, but $\not\models\phi(a_i,b_i)$.
Thus, equationality implies stability \cite{PiSr84}. In stable
theories, \emph{non-forking} provides a natural notion of
independence. Working inside a sufficiently saturated model, we say
that two sets $A$ and $B$ are independent over a common subset $C$,
denoted by $A\ind_C B$, if, for every finite tuple $a$ in $A$, the
type $\tp(a/B)$ does not fork over $C$. Non-forking extensions of a
type over an elementary substructure $M$ to any set $B\supset M$ are
both heir and definable over $M$.

\begin{definition}\label{D:heir_def}

  A type $q$ over $B$ is an \emph{heir} of its restriction
  $q\!\restriction\!M$ to the elementary substructure $M$ if, whenever
  the formula $\phi(x,m,b)$ belongs to $q$, with $m$ in $M$ and $b$ in
  $B$, then there is some $m'$ in $M$ such that $\phi(x,m,m')$ belongs
  to $q\!\restriction\!M$.

  \noindent A type $q$ over $B$ is definable over $M$ if, for each
  formula $\phi(x,y)$, there is a formula $\theta(y)$ with
  parameters in $M$ such that for every $b$ in $B$,
$$\phi(x,b) \in q \text{ if and only if } \models \theta(b).$$
\end{definition}

\noindent Observe that if $q$ is definable over $M$, for any formula
$\phi(x,y)$, any two such formulae $\theta(y)$ are equivalent modulo
$M$, so call it the \emph{$\phi$-definition} of $q$.

If $\phi$ is an equation, the $\phi$-definition of a type $q$ over $B$ is
particularly simple. The intersection
\[\bigcap\limits_{\phi(x,b)\in q} \phi(\mathbb U,b)\]
is a definable set given by a formula $\psi(x)$ over $B$ contained in
$q$. If suffices to set
\[\theta(y) = \forall x\left( \psi(x) \rightarrow
\phi(x,y)\right).\]

By the above characterisation, a formula $\phi(x;y)$ is an equation
if and only if every instance $\phi(a,y)$ is \icl definable sets
\cite[Theorem 3.16]{mJdL01}. A definable set is \emph{\icl} if,
whenever $(b_i)_{i\leq \omega}$ is an indiscernible sequence such that
$b_i$ lies in $X$ for $i < \omega$, then so does $b_\omega$.

\noindent Extending the indiscernible sequence so that it becomes a Morley
sequence over an initial segment, we conclude the following:

\begin{lemma}\label{L:icl_model}
In a complete stable theory $T$, a definable set $\phi(a,y)$ is
\icl if, for every elementary substructure $M$ and every Morley sequence
$(b_i)_{i\leq \omega}$ over $M$ such that

\[ a\ind_M b_i \text{ with } \models \phi(a,b_i) \text{ for } i
<\omega,\]
then $b_\omega$ realises  $\phi(a,y)$ as well.

\noindent We may take the sequence of length $\kappa+1$, for every
infinite cardinal $\kappa$, and assume that
$ a\ind_M \{b_i\}_{i<\kappa}$.
\end{lemma}

In \cite[Theorem 2.5]{gS88}, Srour stated a different criterion for
the equationality of a formula. Let us provide a version of his result.
Given a formula $\phi(x,y)$ and a type $p$ over $B$, denote

\[ p_\phi^+= \{ \phi(x,b)\mid\phi(x,b) \in p\}.\]

\begin{lemma}\label{L:Srour}
  Given a formula $\phi(x;y)$ in a stable theory $T$, the following
  are equi\-valent:
  \begin{enumerate}
  \item\label{L:Srour_eq} The formula $\phi(x;y)$ is an equation.
  \item\label{L:Srour_fteset} Given a tuple $a$ of length $|x|$ and a
    subset $B$, there is a finite subset $B_0$ of $B$ such that
    \[\tp_\phi^+ (a/B_0) \proves \tp_\phi^+(a/B).\]
  \item\label{L:Srour_ind} There is a regular cardinal $\kappa > |T|$
    such that, for any tuple $a$ of length $|x|$ and any elementary
    substructures $M\subset N$ with $a\ind_M N$ and $|N|=\kappa$,
    there is a subset $B_0$ of $N$ with $|B_0|<\kappa$ such that
    \[ \tp(a/MB_0) \proves \tp_\phi^+(a/N).\]
  \end{enumerate}
\end{lemma}

\begin{proof}

For $(\ref{L:Srour_eq})\Longrightarrow (\ref{L:Srour_fteset})$, we
observe that the intersection

\[\bigcap\;\{\phi(\mathbb U,b)\mid\phi(x,b)\in\tp_\phi^+(a/B)\}\]

\noindent is a finite intersection with parameters in a finite subset
$B_0$ of $B$. The implication $(\ref{L:Srour_fteset})\Longrightarrow
(\ref{L:Srour_ind})$ is immediate. For
$(\ref{L:Srour_ind})\Longrightarrow (\ref{L:Srour_eq})$, it suffices
to show that the set $\phi(a,y)$ is \icl, for every tuple $a$ of
length $|x|$. By Lemma \ref{L:icl_model}, let $M$ be an elementary
substructure and $(b_i)_{i\leq \kappa}$ a Morley sequence over $M$
such that
\[ a\ind_M (b_i)_{i<\kappa} \text{ and } \models \phi(a,b_i)
\text{ for } i <\kappa.\]

\noindent We construct a continuous chain of elementary substructures
$(N_i)_{i<\kappa}$, each of cardinality at most $\kappa$ containing $M$, such
that:

\begin{itemize}
\item the sequence $(b_j)_{i\leq j\leq \kappa}$ remains indiscernible over
  $N_i$;
\item $b_{<i}$ is contained in $N_i$;
\item $a\ind_M N_i\cup (b_j)_{i\leq j<\kappa} .$
\end{itemize}

\noindent Set $N_0=M$. For $i<\kappa$ limit ordinal, set
$$N_i=\bigcup\limits_{j<i} N_j.$$ Thus, we need only consider
the successor case. Suppose $N_i$ has already been constructed and let
$N_{i+1}$ be an elementary substructure of cardinality at most
$\kappa$ containing $N_i\cup\{b_i\}$ such that

\[ N_{i+1}\ind_{N_i\cup\{b_i\}} a\cup  (b_j)_{i< j\leq \kappa}.\]

\noindent Observe that the sequence $(b_j)_{i< j\leq \kappa}$ remains
indiscernible over $N_{i+1}$. By monotonicity applied to the above
independence, we have that
\[ N_{i+1}\ind_{N_i\cup  (b_j)_{i\leq  j\leq \kappa}} a,\]
\noindent so, by transitivity,
\[ a \ind_M N_{i+1}\cup  (b_j)_{i<  j\leq \kappa},\]
\noindent as desired.

The elementary substructure $N=\bigcup\limits_{i<\kappa} N_i$ has
cardinality $\kappa$. Finite character implies that $a\ind_M N$. By
hypothesis, there is a subset $B_0$ of $N$ of cardinality strictly
less than $\kappa$ such that
\[ \tp(a/MB_0) \proves \tp_\phi^+(a/N).\]
\noindent Regularity of $\kappa$ yields that $B_0\subset N_i$ for some
$i<\kappa$. In particular, the elements $b_i$ and $b_\kappa$ have the
same type over $N_i$, and therefore over $MB_0$. Let $\tilde a$ such
that $\tilde a b_i\equiv_{MB_0} ab_\kappa$. Since
\[\tp_\phi^+(a/N)\supset \{\phi(x,b_j)\}_{j<\kappa},\]
\noindent we conclude that $\models \phi(\tilde a,b_i)$, and thus
$\models \phi(a, b_\kappa)$, as desired.
\end{proof}

\begin{remark}\label{R:Srour}
  Whenever $a\ind_M N$, the type $\tp(a/N)$ is definable with the same
  definition schema as the one of $\tp(a/M)$. In particular, we can
  add a fourth equivalence to Lemma \ref{L:Srour}: the formula
  $\phi(x;y)$ is an equation if and only if,
  whenever $a \ind_M N$, then
   \[ \tp(a/M) \proves \tp_\phi^+(a/N).\]
\end{remark}

We will finish this section with an observation on imaginaries in
equational theo\-ries.

\begin{lemma}\label{L:eqEI}
  Assume that there is a collection $\mathcal F$ of equations, closed under
  finite conjuntions, such that every formula is a boolean
  combination of instances of formulae in $\mathcal F$. If every
  instance of an equation in $\mathcal F$ has a real
  canonical parameter, then the theory has weak elimination of imaginaries.
\end{lemma}

\begin{proof}

Since the theory is stable, it suffices to show that every global type
$q$ has a real canonical base. As in Definition \ref{D:heir_def}, we
need only include the canonical parameters of the
$\varphi$-definition of every formula $\varphi$ in $\mathcal
F$. Observe that the corresponding formula $\psi(x)$ in $q$ is an
instance of a formula in $\mathcal F$.
\end{proof}


\section{Basics on fields}\label{S:fields}

In this section, we will include some basic notions of field theory
and commutative algebra needed in order to prove the equationality of
the theories of fields we will consider later on. We will work inside
somesufficiently large algebraically closed field $\mathbb U$.

Two subfields $L_1$ and $L_2$ are \emph{linearly
disjoint} over a common subfield $F$, denoted by
\[ L_1 \ld_F L_2,\] if, whenever the elements $a_1,\ldots,a_n$ of
$L_1$ are linearly independent over $F$, then they remain so over
$L_2$, or, equivalently, if $L_1$ has a linear basis over $F$ which
is linearly independent over $L_2$.

\noindent Linear disjointness implies algebraic independence and
agrees with the latter whenever the base field $F$ is algebraically
closed. Let us note that linear disjointness is symmetric, and a
transitive relation: If $F\subset D_2\subset L_2$ is a subfield,
denote by $D_2\cdot L_1$ the field generated by $D_2$ and $L_1$. Then
\[L_1\ld_F L_2\] if and only if
\[ L_1\ld_F D_2 \hskip1cm \text{ and }\hskip1cm D_2\cdot
L_1\ld_{D_2} L_2.\]

\noindent By multiplying with a suitable denominator, we may also use
the terminology for a subring $A$ being linearly disjoint from a
field$B$ over a common subring $C$.

\begin{definition}\label{D:spec}
  Consider a theory $T$ of fields in the language $\LL$ extending the
  language of rings $\LL_{rings}=\{+,-,\,\cdot\,,0,1\}$ such that
  there is a predicate $\Pred$, which is interpreted in every model of
  $T$ as a definable subfield. A subfield $A$ of a sufficiently
  saturated model $K$ of $T$ is $\Pred$-\emph{special} if
\[ A\ld_{\Pred(A)} \Pred(K),\] where $\Pred(A)$ equals $\Pred(K)\cap
A$.
\end{definition}
\noindent It is easy to see that elementary substructures of $K$ are
$\Pred$-special.
\begin{lemma}\label{L:coh}
  Inside a sufficiently saturated model $K$ of a stable theory $T$ of
  fields in the language $\LL\supset \LL_{rings}$ equipped with a
  definable subfield $\Pred(K)$, consider a $\Pred$-special
  field $A$ and a field $B$, both containing an elementary
  substructure $M$ of $K$ such that $A\ind_M B$. The
  fields $\Pred(K)\cdot A$ and $\Pred(K)\cdot B$  are linearly disjoint
  over $\Pred(K)\cdot M$.
\end{lemma}
\noindent Note that we write $F\cdot F'$\/ for the field generated by
$F$ and $F'$.
\begin{proof}
  It suffices to show that elements $a_1,\ldots,a_n$ of $A$ which are
  linearly dependent over $\Pred(K)\cdot B$ are also linearly
  dependent over $\Pred(K)\cdot M$. Thus, let $z_1,\ldots,z_n$ in
  $\Pred(K)\cdot B$, not all zero, such that \[ \sum\limits_{i=1}^n
  a_i\cdot z_i=0.\]

\noindent Multiplying by a suitable denominator, we may assume
that all the $z_i$'s lie in the subring generated by $\Pred(K)$ and
$B$, so

\[z_i = \sum\limits_{j=1}^m \zeta^i_j b_j,\] for some $\zeta^i_j$'s
in$\Pred(K)$ and $b_1,\ldots, b_m$ in $B$, which we may assume to be
linearly independent over $\Pred(K)$.

\noindent The type $\tp(a_1,\ldots, a_n/Mb_1,\ldots b_m)$ is a
nonforking extension of $\tp(a_1,\ldots, a_n/M)$, so in particular is
a
heir over $M$.  Thus, there are
some $\eta^i_j$'s in $\Pred(K)$, not all zero, and $c_1,\ldots, c_m$
in $M$ linearly independent over $\Pred(K)$, such that

\[ \sum\limits_{i=1}^n a_i\sum\limits_{j=1}^m \eta^i_j c_j=0.\]

\noindent Since $A$ is $\Pred$-special, we may assume all the
$\eta^i_j$'s lie in $\Pred(A)$.  As $\{c_j\}_{1\leq j\leq m}$ are
$\Pred$-linearly independent, at least one of the elements in
\[ \{ \sum\limits_{1\leq j\leq m} \eta^1_j c_j, \ldots,
\sum\limits_{1\leq j\leq m}
  \eta^n_j c_j \}\]

\noindent is different from $0$, as desired.

\end{proof}

A natural example of a definable subfield is the field of
$p^\text{th}$ powers $K^p$, whenever $K$ has positive characteristic
$p>0$. The corresponding notion of $K^p$-special is
\emph{separability}: A non-zero polynomial $f(T)$ over a subfield $K$
is separable if every root (in the algebraic closure of $K$) has
multiplicity $1$, or equivalently, if $f$ and its formal derivative
$\frac{\partial f}{\partial T}$ are coprime. Whenever $f$ is
irreducible, the latter is equivalent to $\frac{\partial f}{\partial
  T}\neq 0$. In particular, every non-constant polynomial in
characteristic $0$ is separable. In positive characteristic $p$, an
irreducible polynomial $f$ is separable if and only if $f$ is not a
polynomial in $T^p$.

\noindent An algebraic extension $K\subset L$ is \emph{separable} if
the minimal polynomial over $K$ of every element in $L$ is
separable. Algebraic field extensions in characteristic $0$ are always
separable. In positive characteristic $p$, the finite extension is
separable if and only if the fields $K$ and $L^p$ are linearly
disjoint over $K^p$. This explains the following definition:

\begin{definition}\label{D:sep_ext}
  An arbitrary (possibly not algebraic) field extension $F\subset K$
  is \emph{separable} if, either the characteristic is $0$ or, in case
  the characteristic is $p>0$, the fields $F$ and $K^p$ are linearly
  disjoint over $F^p$.
\end{definition}

\noindent A field $K$ is \emph{perfect} if either it has
characteristic $0$ or if $K=K^p$, for $p=\car(K)$. Any field
extension of a perfect field is separable. Given a field $K$, we
define its \emph{imperfection degree} (in $\N\cup\{\infty\})$, as $0$
if the characteristic of $K$ is $0$, or $\infty$, in case of positive
characteristic $p$ if $[K:K^p]$ is infinite. Otherwise $[K:K^p]=p^e$
for $e$ the degree of imperfection. Thus, a field is perfect if and
only if its imperfection degree is $0$

Another example of fields equipped with a definable
subfields are differential fields. A differential field consists of
a field $K$ together with a distinguished additive morphism $\delta$
satisfying Leibniz' rule

\[ \delta(xy)=x\delta(y)+y\delta(x).\]

\noindent Analogously to Zariski-closed sets for pure field, one
defines
\emph{Kolchin-closed} sets in differential fields as zero sets of
systems of differential-polynomials equations, that is, polynomial
equations on the different iterates of the variables under the
derivation. For a tuple $x=(x_1,\ldots,x_n)$ in $K$,  denote by
$\delta(x)$ the tuple $(\delta(x_1),\ldots,\delta(x_n))$.

\begin{lemma}\label{L:eq_allchar_delta}
  In any differential field $(K,\delta)$, an algebraic differential
  equation \[p(x,\delta x,\delta^2x\ldots;y,\delta
  y,\delta^2y,\ldots)\doteq 0\] is an equation in the model-theoretic
  sense.
\end{lemma}
\begin{proof}
  In characteristic zero this follows from Ritt-Raudenbush's
  Theorem, which states that the Kolchin topology is noetherian. In
  arbitrary characteristic, it suffices to observe, as in Example
  \ref{E:hilbert}, that $p(x,\delta x,\ldots;y,\delta y,\ldots)$ can
  be written as $q(M_1,\ldots;y, \delta y,\ldots)$ where
  $q(u_1,\dotsc;y_0,y_1,\ldots)$ is linear in the $u_i$'s, for some
  differential monomials $M_j$'s in $x$.
\end{proof}

In particular, the theory $\mathsf{DCF}_0$ of differentially closed
fields of
characteristic $0$ is equational, since it has quantifier elimination
\cite{cW98}.

\noindent In a differential field $(K,\delta)$, the set of
\emph{constants}
\[ \C_K=\{x\in K\mid\delta(x)=0\}\] is a definable subfield, which
contains $K^p$ if $p=\car(K)>0$. If $K$ is algebraically
closed, then so is $\C_K$.

\begin{fact}\label{F:Wronski}
The elements $a_1,\ldots, a_k$ of the differential field $(K,\delta)$
are linearly dependent over $\C_K=\{x\in K\mid\delta(x)=0\}$ if and
only if their \emph{Wronskian} $\W(a_1,\ldots, a_k)$ is $0$, where

\[
\W(a_1,\ldots, a_k)=\det\begin{pmatrix}
 a_1 & a_2 & \ldots & a_k \\
 \delta(a_1) &  \delta(a_2) & \ldots &  \delta(a_k) \\
 \vdots & & \vdots & \\
  \delta^{k-1}(a_1) &  \delta^{k-1}(a_2) & \ldots &
\delta^{k-1}(a_k)
\end{pmatrix}.\]
\end{fact}

\noindent Whether the above matrix has determinant $0$ does not
dependon the differential field where we compute it. In particular,
every differential subfield $L$ of $K$ is $\C_K$-special.

Perfect fields of positive characteristic cannot have non-trivial
derivations. In characteristic zero though, any field $K$ which is
notalgebraic over the prime field has a non-trivial derivation
$\delta$. Analogously to perfectness, we say that the differential
field $(K,\delta)$ is \emph{differentially perfect} if either $K$ has
characteristic $0$ or, in case $p=\car(K)>0$, if every constant has a
$p^{th}$-root, that is, if $\C_K=K^p$.

\noindent Notice that the following well-known result generalises the
equivalent situation for perfect fields and separable extensions.

\begin{remark}\label{R:diffperf_sep}
  Let $(K,\delta)$ be a differential field and $F$ a differentially
  perfect differential subfield of $K$. The extension $F\subset K$ is
  separable.
\end{remark}

\begin{proof}
  We need only prove it when the characteristic of $K$ is $p>0$. By
Fact \ref{F:Wronski}, the fields $F$ and $\C_K$ are linearly disjoint
over
  $\C_F=F^p$. Since $K^p\subset\C_K$, this implies that $F$ and $K^p$
  are linearly disjoint over $F^p$.
\end{proof}

In section \ref{S:MTPairs}, we will consider a third theory of fields
equipped with a definable subfield: belles paires of algebraically
closed fields. In order to show that the corresponding theory is
equational, we require some basic notions from linear algebra (\cf
\cite[R\'esultats d'Alg\`ebre]{jD74}). Fix some subfield $E$ of
$\mathbb U$.

Let $V$ be a vector subspace of $E^n$ with basis
$\{v_1,\ldots,v_k\}$. Observe that

\[V=\bigl\{v\in E^n\,\bigm|\, v\wedge (v_1\wedge\cdots\wedge v_k)=0
\text{ in
} \extp^{k+1} E^n\bigr\}.\]

\noindent The vector $v_1\wedge\cdots\wedge v_k$ depends only on $V$,
up to scalar multiplication, and determines $V$ completely.  \emph{The
  Pl\"ucker coordinates} $\pk(V)$ of $V$ are the homogeneous
coordinates of $v_1\wedge\cdots\wedge v_k$ with respect to the
canonical basis of $\extp^{k} E^n$. The
\emph{$k^\text{th}$-Grassmannian} $\Gr_k(E^n)$ of $E^n$ is the
collection of Pl\"ucker coordinates of all $k$-dimensional subspaces
of $E^n$.  Clearly $\Gr_k(E^n)$ is contained in $\Proj^{r-1}(E)$, for
$r={n \choose k}$.

\noindent The $k^\text{th}$-Grassmannian is Zariski-closed. Indeed,
given an
element  $\zeta$ of $\extp^k E^n$,  there is a smallest vector
subspace $V_\zeta$ of $E^n$ such that $\zeta$ belongs to $\extp^k
V_\zeta$.
The vector space $V_\zeta$ is the collection of inner products $e
\pluck
\zeta$, for $e$ in $\extp^{k-1} (E^n)^*$. Recall that  the
inner product $\pluck$ is a bilinear map
\[ \pluck: \extp^{k-1} (E^n)^* \times \extp^{k} (E^n) \to E. \]
\noindent A non-trivial element $\zeta$ of $\extp^k E^n$
determines a  $k$-dimensional subspace of $E^n$ if and only if
\[\zeta \wedge (e\pluck \zeta)=0,\]
\noindent for every $e$ in $\extp^{k-1} (E^n)^*$. Letting $e$ run over
a fixed basis of $\extp^{k-1} (E^n)^*$, we see that the
$k^\text{th}$-Grassmannian is the zero-set of a finite collection of
homogeneous polynomials.

Let us conclude this section with an observation regarding projections
of
certain varieties.

\begin{remark}\label{R:proj_complete}
  Though the theory of algebraically closed fields has elimination of
  quantifiers, the projection of a Zariski-closed set need not be
  again closed. For example, the closed set
  $$ V=\{(x,z) \in E \times E\mid x\cdot z=1\}$$ projects onto the
  open set $\{x \in E\,| \, x\neq 0\}$.  An algebraic variety $Z$ is
  \emph{complete} if, for all varieties $X$, the projection
$X \times Z\to X$ is a Zariski-closed map.  Projective varieties are  complete.
\end{remark}

\section{Model Theory of separably closed fields}\label{S:MTSCF}

Recall that a field $K$ is \emph{separably closed} if it has no proper
algebraic separable extension, or equivalently, if every non-constant
separable polynomial over $K$ has a root in $K$. For each fixed
degree, this can be expressed in the language of rings. Thus, the
class of separably closed fields is axiomatisable. Separably closed
fields of characteristic zero are algebraically closed. For a prime
$p$, let $\scfp$ denote the theory of separably closed fields of
characteristic $p$ and $\scfpe$ the theory of separably closed fields
of characteristic $p$ and imperfection degree $e$. Note that
$\mathsf{SCF}_{p,0}$ is the theory $\mathsf{ACF}_p$ of algebraically
closed fields of characteristic $p$.

\begin{fact}\label{F:Delon} (cf.~\cite[Proposition 27]{fD88})
  The theory $\scfpe$ is complete and stable, but not superstable for
  $e>0$. Given a model $K$ and a separable field extension
  $k\subset K$, the type of $k$ in $K$ is completely determined by its
  quantifier-free type. In particular, the theory has quantifier
  elimination in the language
  \[\LL_{\lambda}=\LL_{rings}\cup\{\lambda_n^i\mid 1\leq i\leq
    n<\omega \},\] where the value $\lambda_n^i(a_0,\ldots,a_n)$ is
  defined as follows in $K$. If there is a \emph{unique} sequence
  $\zeta_1,\dotsc,\zeta_n\in K$ with
  $a_0 = \zeta_1^p \, a_1+\dotsb+\zeta_n^p\,a_n$, we set
  $\lambda_n^i(a_0,\ldots,a_n)=\zeta_i$. Otherwise, we set
  $\lambda_n^i(a_0,\ldots,a_n)=0$ and call it \emph{undefined},
\end{fact}

Note that $\lambda_n^i(a_0,\ldots,a_n)$ is defined if and only if
\[K\models \neg \pdep_n(a_1,\ldots,a_n) \land
\pdep_{n+1}(a_0,a_1,\ldots,a_n),\] where $\pdep_n(a_1,\ldots,a_n)$
means that $a_1,\dotsc,a_n$ are $K^p$-linearly dependent. In
particular, the value
$\lambda_n^i(a_0,\ldots,a_n)$ is undefined for $n>p^e$.

For a subfield $k$ of a model $K$ of $\scfp$, the field extension
$k\subset K$ is separable if and only if $k$ is closed under
$\lambda$-functions.

\begin{notation}
  For elements $a_0,\ldots,a_n$ of $K$, the notation
  $\ldef{a_0, a_1,\ldots,a_n}$ is an abbreviation for
  $\neg \pdep_n(a_1,\ldots,a_n) \land \pdep_{n+1}(a_0,a_1,\ldots,a_n)$.
\end{notation}

\begin{remark}\label{R:fte_e_EI}
  If the imperfection degree $e$ is finite, we can fix a $p$-basis
  $\mathbf{b}=(b_1,\dotsc,b_e)$ of $K$, that is, a tuple such that the
  collection of monomials
  $$\mathbf{\bar b}=(b_1^{\nu_1}\dotsb b_e^{\nu_e}\,\mid\, 0\leq
  \nu_1,\dotsc,\nu_e<p)$$
  \noindent is a linear basis of $K$ over $K^p$. All $p$-bases have
  the same type. If we replace the $\lambda$-functions by the
  functions $\Lambda^\nu(a)=\lambda_{p^e}^{\nu}(a,\mathbf{\bar b})$,
  then the theory $\scfpe(\mathbf{b})$, in the language of
  rings with constants for $\mathbf{b}$ and equipped with the
  functions $\Lambda^\nu(x)$, has again quantifier
  elimination. Furthermore, the $\Lambda$-values of a sum or a product
  can be easily computed in terms of the values of each factor. In
  particular, the canonical base of the type $(a/K)$ in
  $\scfpe(\mathbf{b})$ is the field of definition of the
  vanishing ideal of the infinite tuple
  \[ (a,\overline{\Lambda}(a),
  \overline{\Lambda}(\overline{\Lambda}(a)),\ldots).\]
\noindent Thus, the theory $\scfpe(\mathbf{b})$ has
elimination of imaginaries.

\noindent As in Lemma \ref{L:eq_allchar_delta}, it follows that the
formula $t(x;y)\doteq 0$ is a model-theoretic equation, for every
$\LL_{\Lambda}$-term $t(x,y)$. This implies that
$\scfpe(\mathbf{b})$, and therefore $\scfpe$, is equational.

 Whether there is an explicit expansion of the language of rings in which
  $\scfpi$ has elimination of imaginaries is not yet
  known.
\end{remark}

From now on, work inside a sufficiently saturated model $K$ of the
incomplete theory $\scfp$. The imperfection degree of $K$ may be
either finite or infinite.

\noindent Since an $\LL_\lambda$-substructure determines a separable field
extension, Lemma \ref{L:coh} implies the following result:

\begin{cor}\label{C:coh_SCFinfty}
  Consider two subfields $A$ and $B$ of $K$ containing an elementary
  substructure $M$ of $K$. Whenever \[A\ind^{\scfp}_M B,\]
  \noindent the fields $K^p\cdot A$ and $K^p\cdot B$
  are linearly disjoint over $K^p\cdot M$.
\end{cor}
\noindent Note that the field $K^p\cdot A$ is actually the ring
generated by $K^p$ and $A$, since $A$ is algebraic over $K^p$.
\begin{proof}
  The $\LL_{\lambda}$-structure $A'$ generated by $A$ is a subfield,
  since $a\inv=\lambda_1^1(1,a^p)$ for $a\not=0$. Since
  $A'\ind^{\scfpe}_M B$, and $A'$ is $K^p$-special, we
  have that $K^P\cdot A'$ and $K^p\cdot B$ are linearly disjoint over
  $M$. Whence $K^P\cdot A$ and $K^P\cdot B$ are also linearly disjoint
  over $M$
\end{proof}

We will now exhibit our candidate formulae for the equationality of
$\scfp$, uniformly on the imperfection degree.

\begin{definition}\label{D:lambda_tame}
  The collection of $\lambda$-tame formulae is the smallest collection
  of formulae in the language $\LL_{\lambda}$, containing all
  polynomial equations and closed under conjunctions, such that, for
  any natural number $n$ and polynomials $q_0, \ldots, q_n$ in
  $\Z[x]$, given a $\lambda$-tame formula $\psi(x, z_1,\ldots, z_n)$,
  the formula
  \begin{multline*}
    \phi(x)=\pdep_n(q_1(x),\ldots, q_n(x))\;\;\lor\;\;\\
    \bigl(\,\ldef{q_0(x),\ldots, q_n(x)}\;\land\;\, \psi(x,
    \overline{\lambda}_n(q_0(x), \ldots, q_n(x)))\,\bigr)
  \end{multline*}
  is $\lambda$-tame.
\end{definition}
\noindent Note that the formula $\phi$ above is equivalent to
\[\pdep_n(q_1,\ldots, q_n)\;\lor\;
\bigl(\pdep_{n+1}(q_0,\ldots, q_n)\,\land\, \psi(x,
\overline{\lambda}_n(\overline{q}(x)))\bigr).\]

\noindent In particular, the formula $\pdep_n(q_1(x),\ldots, q_n(x))$ is a tame
$\lambda$-formula, since it is equivalent to

\[\pdep_n(q_1(x),\ldots, q_n(x))\;\;\lor\;\;\\
    \bigl(\,\ldef{0,q_1(x),\ldots, q_n(x)}\;\land\;\,0\doteq 1\bigr).\]

There is a natural degree associated to a $\lambda$-tame formula, in
terms of the amount of nested $\lambda$-tame formulae it contains,
whereas polynomial equations have degree $0$. The degree of a
conjunction is the maximum of the degrees of the corresponding
formulae.

The  next remark is easy to prove by induction on the degree of the
formula:

\begin{remark}\label{R:lambda_transf}
  Given a $\lambda$-tame formula $\phi$ in $m$ many free variables and
  polynomials $r_1(X), \ldots, r_m(X)$ in several variables with
  integer coefficients, the formula $\phi(r_1(x),\ldots,r_m(x))$ is
  equivalent in $\scfp$ to a $\lambda$-tame formula of the same
  degree.
\end{remark}

\begin{prop}\label{P:lambda_BK}
  Modulo $\scfp$, every formula is equivalent to a Boolean
  combination of $\lambda$-tame formulae.
\end{prop}

\begin{proof}
  By Fact \ref{F:Delon}, it suffices to show that the equation
  $t(x)\doteq0$ is equivalent to a Boolean combination of
  $\lambda$-tame formulae, for every $\LL_{\lambda}$-term
  $t(x)$. Proceed by induction on the number of occurrences of
  $\lambda$-functions in $t$. If no $\lambda$-functions occur in
  $t$, the result follows, since polynomial equations are
  $\lambda$-tame. Otherwise
  \[t(x)=r(x,\overline\lambda_n(q_0(x),\dotsc,q_n(x)))\] for some
  $\LL_{\lambda}$-term $r(x,z_1,\dotsc,z_n)$ and polynomials $q_i$. By
  induction,  the term $r(x,\bar z)\doteq0$ is equivalent
  to a Boolean combination
  $BK(\psi_1(x,\bar z),\dotsc,\psi_m(x,\bar z))$ of $\lambda$-tame
  formulae $\psi_1(x,\bar z),\dotsc,\psi_m(x,\bar z)$. Consider now
  the $\lambda$-tame formulae
  \[\phi_i(x)=\pdep_n(q_1(x),\ldots, q_n(x))\;\lor\;
  \bigl(\,\ldef{\overline{q}(x)}\;\land\;\, \psi_i(x,
  \overline{\lambda}_n(\overline{q}(x)))\bigr).\] Note that
  \[ \scfpe\models \Bigl( (\ldef{\overline{q}(x)}) \longrightarrow (\psi_i(x,
  \overline{\lambda}_n(\overline{q}(x)))\,\leftrightarrow\,\phi_i(x)) \Bigr).\]

\noindent Therefore $t(x)\doteq0$ is equivalent to
  \[\bigl(\neg\ldef{\overline{q}(x)}\,\land\;r(x,0)\doteq0\bigr)\;\;\lor
  \;\;\bigl(\ldef{\overline{q}(x)}\,\land\;
  BK(\phi_1(x),\dotsc,\phi_m(x))\bigr),\] which is, by induction, a
  Boolean combination of $\lambda$-tame formulae.
\end{proof}

We conclude this section with a homogenisation result for
$\lambda$-tame formulae, which will be used in the proof of the
equationality of $\scfp$.

\begin{prop}\label{P:lambda_hom}
  For every $\lambda$-tame $\phi(x,y_1,\ldots, y_n)$ there is a
  $\lambda$-tame formula $\phi'(x,y_0,y_1,\dotsc,y_n)$ of same
  degree such that
  \[  \scfp\models\forall x,y_0 \dotsc y_n \Bigl(
  \phi'(x,y_0,\ldots, y_n)\longleftrightarrow \Bigl(
  \phi\Bigl(x,\frac{y_1}{y_0},\dotsc, \frac{y_n}{y_0}\Bigr) \lor
  y_0 \doteq 0 \Bigr) \Bigr) .\]
\end{prop}
\noindent We call $\phi'$ a \emph{homogenisation} of $\phi$ with
respect to $y_0,\ldots, y_n$.

\begin{proof}

  Let $y$ denote the tuple $(y_1,\dotsc,y_n)$. By induction on the
  degree, we need only consider basic $\lambda$-tame formulae, since
  the result is preserved by taking conjunctions. For degree $0$,
  suppose that $\phi(x,y)$ is the formula $q(x,y)\doteq 0$, for some
  polynomial $q$.  Write \[q(x,\frac{y}{y_0})=\frac{q'(x,y_0,y)}{y_0^N}.\] Then
  $\phi'(x,y_0,y)=y_0\cdot q'(x,y)\doteq0$ is a homogenisation.

\noindent  If $\phi(x,y)$ has the form
  \[ \pdep_n(q_1(x,y),\dotsc, q_m(x,y)) \lor
  \left( \ldef{q_0,\dotsc, q_m} \land\, \psi(x,y,
  \overline{\lambda}_n(q_0, \dotsc, q_m)) \right), \] let
  $\psi'(x,y_0,y,z)$ be a homogenisation of $\psi(x,y,z)$ with respect
  to $y_0,y$. There is a natural number $N$ such that for each $0\leq
  j\leq m$,
  \[q_j(x,\frac{y}{y_0})=\frac{q'_j(x,y_0,y)}{y_0^N}\]
  for polynomials $q'_j$. Set now $q''_j=y_0\cdot q'_j$ and
  \[\phi'(x,y_0,y) = \pdep_n(q''_1,\dotsc, q''_m) \lor \left(
  \ldef{q''_0,\ldots, q''_m} \land\, \psi'(x,y_0,y,
  \overline{\lambda}_n(q''_0, \dotsc, q''_m)) \right).\]
\end{proof}
\section{Equationality of $\scfp$}\label{S:EqSCFp_alt}

By Proposition \ref{P:lambda_BK}, in order to show that the theory
$\scfp$ is equational, we need only show that each $\lambda$-tame
formula is an equation in every completion $\scfpe$. As before, work
inside a sufficiently saturated model $K$ of some fixed imperfection
degree.

For the proof, we require \emph{generalised
  $\lambda$-functions}: If the vectors $\bar a_0,\dotsc,\bar a_n$ in
$K^N$ are linearly independent over $K^p$ and the system
\[\bar a_0=\sum\limits_{i=1}^n \zeta_i^p \, \bar a_i\]
has a solution, then it is unique and denoted by
$\lambda^i_{N,n}(\bar a_0,\dotsc,\bar a_n)$. The notation
$\ldef[_{N,n}]{\bar a_0,\dotsc,\bar a_n}$ means that all
$\lambda_{N,n}^i$'s are defined. Observe that
$\lambda^i_{1,n}=\lambda_n^i$.  We denote by
$\pdep_{N,n}(\bar a_0,\dotsc,\bar a_n)$ the formula stating that the
vectors $\bar a_1,\dotsc,\bar a_n$ are linearly dependent over $K^p$.

\begin{theorem}\label{T:eq_SCF}
  Given any partition of the variables, every $\lambda$-tame formula
  $\phi(x;y)$ is an equation in $\scfpe$
\end{theorem}

\begin{proof}
  We proceed by induction on the degree $D$ of the $\lambda$-tame
  formula. For $D=0$, it is clear. So assume that the theorem is true
  for all $\lambda$-tame formulae of degree smaller than some fixed
  degree $D \geq 1$. Let $\phi(x;y)$ be a $\lambda$-tame
  formula of degree $D$.
  \begin{claimstern} If
    \begin{multline*}
    \phi(x;y)=\pdep_{N,n}(\bar q_1(x^p,y),\ldots,
    \bar q_n(x^p,y))\;\;\lor\;\;\\ \bigl(\,\ldef[_N]{\bar q_0(x^p,y),\ldots,
      \bar q_n(x^p,y)}\;\land\;\, \psi(x,y,
    \overline{\lambda}_{N,n}(\bar q_0(x^p,y),
    \ldots,\bar q_n(x^p,y)))\,\bigr),
  \end{multline*}
    where $\psi(x, y, z_1,\ldots,z_n)$ is a $\lambda$-tame formula of
    degree $D-1$, then $\phi(x;y)$ is an equation.
  \end{claimstern}
  \begin{claimsternproof} It suffices to show
    that every instance $\phi(x,b)$ is equivalent to a formula
    $\psi'(x,b',b)$, where $\psi'(x,y',y)$ is a $\lambda$-tame formula
    of degree $D-1$, for some tuple $b'$.

    \noindent Choose a $K^p$-basis $b_1,\dotsc,b_{N'}$ of all
    monomials in $b$ occurring in the $\bar q_k(x^p,b)$'s and write
    $\bar q_k(x^p,b)=\sum_{j=1}^{N'} \bar q_{j,k}(x,b')^p\cdot
    b_j$. We use the notation $\mathbf{q}_k(x,b')$ for the vector of
    length $NN'$ which consists of the concatenation of the vectors
    $\bar q_{j,k}(x,b')$. Let $\mathbf{Q}(x,b')$ be the
    $(NN'\times n)$-matrix with columns
    $\mathbf{q}_1(x,b'),\dotsc,\mathbf{q}_n(x,b')$. The vectors
    $\bar q_1(x^p,b),\ldots, \bar q_n(x^p,b)$ are linearly dependent
    over $K^p$ if and only if the columns of $\mathbf{Q}(x,b')$
    are linearly dependent over $K$. Let $J$ range over all $n$-element
    subsets of $\{1,\dotsc,NN'\}$ and let $\mathbf{Q}^J(x,b')$ be the
    corresponding $n\times n$-submatrices. Thus
\[ \scfpe \models \Bigl( \pdep_{N,n}(\bar q_1(x^p,y),\ldots,
    \bar q_n(x^p,y)) \longleftrightarrow
    \Land_J\det(\mathbf{Q}^J(x,b'))\doteq 0 \Bigr).\]

\noindent If $\det(\mathbf{Q}^J(x,b'))$ is not zero, the vector
    $\overline\zeta=\overline{\lambda}_{N,n}(\bar q_0(x^p,b), \ldots,
    \bar q_n(x^p,b))$ is defined if and only if $\mathbf
    q_0(x,b')=\mathbf{Q}(x,b')\cdot\overline\zeta$. In that case,
    \[\overline\zeta=\det(\mathbf{Q}^J(x,b'))\inv\cdot B^J(x,b')\cdot
    \mathbf q_0^J(x,b'),\] where $B^J(x,b')$ is the adjoint of
    $\mathbf{Q}^J(x,b')$. Set
    $d^J(x,b')=\det(\mathbf{Q}^J(x,b'))$ and $r^J(x,b')= B^J(x,b')\cdot
    \mathbf q_0^J(x,b')$, so
    \[\overline\zeta=d^J(x,b')\inv\cdot r^J(x,b').\]

    \noindent Consider the $\lambda$-tame formula
    \[\psi^J(x,b',b,\overline z)\;\;=\;\;
    \bigl(\mathbf q_0(x,b')\doteq \mathbf{Q}(x,b')\cdot\overline
    z\;\land\; \psi(x,b,\overline z)\bigr),\]
\noindent  of degree $D-1$.  It follows that
    $\phi(x,b)$ is equivalent to
    \[\Land_J\;\; \bigl(d^J(x,b')\doteq 0\;\;\lor
    \;\;\psi^J(x,b',b,d^J(x,b')\inv\cdot r^J(x,b'))\bigr),\] which is
    equivalent to a $\lambda$-tame formula of degree $D-1$, by Remark
    \ref{R:lambda_transf} and Proposition \ref{P:lambda_hom}.
  \end{claimsternproof}

  For the proof of the theorem, since a conjunction of equations is
  again an equation, we may assume that
  \begin{multline*}
    \phi(x;y)=\pdep_n(q_1(x,y),\ldots,
    q_n(x,y))\;\;\lor\;\;\\ \bigl(\,\ldef{q_0(x,y),\ldots,
      q_n(x,y)}\;\land\;\, \psi(x,y, \overline{\lambda}_n(q_0(x,y),
    \ldots, q_n(x,y)))\,\bigr)
  \end{multline*}
  for some $\lambda$-tame formula $\psi(x, y, z_1,\ldots,z_n)$ of
  degree $D-1$. It suffices to show that $\varphi(a,y)$ is \icl. By
  Lemma \ref{L:icl_model}, consider an elementary substructure $M$ of
  $K$ and a Morley sequence $(b_i)_{i\leq \omega}$ over $M$ such that
  \[ a\ind_M b_i \text{ with } \models \phi(a,b_i) \text{ for } i
  <\omega.\] We must show that $K\models\phi(a,b_\omega)$.

\noindent  Choose a $(K^p\cdot M)$-basis $a_1,\dotsc,a_N$ of the monomials in
  $a$ which occur in the $q_k(a,y)$ and write $q_k(a,y)=\sum_{j=1}^N
  q_{j,k}(a'^p,m,y)\cdot a_j$, for some tuple $m$ in
  $M$ and $a'$ in $K$. Let $\bar q_k(a'^p,m,y)$ be the
  vector $\bigl(q_{j,k}(a'^p,m,y)\bigr)_{1\leq j\leq N}$ and consider
  the formula
  \begin{multline*}
    \phi'(x,x';y',y)=\pdep_{N,n}(\bar q_1(x'^p,y',y),\ldots, \bar
    q_n(x'^p,y',y))\;\;\lor\;\;\\ \bigl(\,\ldef[_N]{\bar
      q_0(x'^p,y',y),\ldots, \bar q_n(x'^p,y',y)}\;\land\;\\
    \psi(x,y,
    \overline{\lambda}_{N,n}(\bar q_0(x'^p,y',y), \ldots,\bar
    q_n(x'^p,y',y)))\,\bigr).
  \end{multline*}

 \noindent Clearly,
\[ \scfpe \models \forall y (\phi'(a,a',m,b) \longrightarrow
  \phi(a,y)).\]
\noindent By Corollary \ref{C:coh_SCFinfty},
  the elements $a_1,\dotsc,a_N$ are linearly independent over the field
  $(K^p\cdot M)(b_i)$, so $\phi'(a,a',m,b_i)$ holds in $K$, since
  $K\models \phi(a,b_i)$ for $i<\omega$.
  By the previous claim, the $\lambda$-tame formula
  $\phi'(x,x';y',y)$ is an equation. Since the sequence $(m,b_0),\ldots
  (m,b_\omega)$ is indiscernible, we have that
  $\phi'(a,a',m,b_\omega)$ holds in $K$, so $K\models
  \phi(a,b_\omega)$, as desired.
\end{proof}

Together with Proposition \ref{P:lambda_BK}, the above theorem yields
the following:

\begin{cor}\label{C:SCF_eq}
  The (incomplete) theory $\scfp$ of separably closed fields of characteristic
  $p>0$ is equational.
\end{cor}
\begin{proof}
  Proposition \ref{P:lambda_BK} yields, that modulo $\scfp$ every
  formula is a Boolean combination of sentences (i.e.\ formulas
  without free variables) and $\lambda$-tame formulas. Sentences are
  equations by definition, $\lambda$-tame formulas are equations by
  Theorem \ref{T:eq_SCF}.
\end{proof}

Lemma \ref{L:eqEI} and Theorem \ref{T:eq_SCF} yield a partial
elimination of imaginaries for $\scfpe$.

\begin{cor}\label{C:EI_SCF}
The theory $\scfpe$ of separably closed fields of characteristic
  $p>0$ and imperfection degree $e$ has weak elimination of
  imaginaries, after adding canonical parameters for all instances of
  $\lambda$-tame formulae.
\end{cor}

\begin{question}
Is there an explicit description of the canonical parameters of instances of
 $\lambda$-tame formulae, similar to the geometric sorts introduced
 in \cite{aP07}?
\end{question}

\section{Model Theory of differentially closed fields in positive
  characteristic}\label{S:MTDCFp}

The model theory of existentially closed differential fields in
positive characteristic has been thoroughly studied by Wood
\cite{cW73, cW76}.  In contrast to the characteristic $0$ case, the
corresponding theory is no longer $\omega$-stable nor superstable:
its universe is a separably closed field of infinite
imperfection degree (see Section \ref{S:MTSCF}).

\noindent A differential field $(K,\delta)$ is \emph{differentially
  closed} if it is existentially closed in the class of differential
fields. That is, whenever a quantifier-free
$\LL_\delta=\LL_{rings}\cup\{\delta\} $-formula $\phi(x_1,\ldots,
x_n)$, with parameters in $K$, has a realisation in a differential field
extension $(L,\delta_L)$ of $(K,\delta)$, then there is a realisation of
$\phi(x_1,\ldots,x_n)$ in $K$.

A \emph{differential polynomial} $p(x)$ is a polynomial in $x$ and its
higher order derivatives $\delta(x),\delta^2(x),\dotsc$ The
\emph{order} of $p$ is the order of the highest occurring derivative.

\begin{fact}\label{F:Wood}
  The class of differentially closed fields of positive characteristic
  $p$ can be axiomatised by the complete theory $\dcfp$ with
  following axioms:
\begin{itemize}
\item The universe is a differentially perfect differential field of
  characteristic $p$.
\item Given two differential polynomials $g(x)\neq 0$ and $f(x)$ in
  one variable with $\ord(g)<\ord(f)=n$ such that the \emph{separant}
  $s_f=\frac{\partial f}{\partial (\delta^nx)}$ of $f$ is not
  identically $0$, there exists an element $a$ with $g(a)\neq 0$ and
  $f(a)=0$.
\end{itemize}

\noindent The type of a differentially perfect differential subfield
is determined by its quantifier-free type. The theory $\dcfp$ is
stable but not superstable, and has quantifier-elimination in the
language $\LL_{\delta,s}=\LL_\delta\cup\{s\}$, where $s$ is the
following unary function:

\[ s(a)= \begin{cases} b, \text{ with } a=b^p \text{ in case }
    \delta(a)=0. \\
0, \text{ otherwise.} \end{cases} \]
\end{fact}

\noindent Note that every non-constant separable polynomial is a
differential polynomial of order $0$ whose separant is non-trivial
(since $\delta^0(x)=x$). In particular, every model $K$ of $\dcfp$ is
a separably closed field. Furthermore, the imperfection degree of $K$
is infinite: Choose for every $n$ in $\N$ an element $a_n$ in $K$ with
$\delta^n(a_n) =0$ but $\delta^{n-1}(a_n)\neq 0$. It is easy to see
that the family $\{a_n\}_{n\in\N}$ is linearly independent over $K^p$.

\begin{remark}\label{R:closure_s}
  The quotient field of any $\LL_{\delta,s}$-substructure of a model of
  $\dcfp$ is differentially perfect.
\end{remark}
\begin{proof}
  Let $\frac{a}{b}$ be an element in the quotient field with
  derivative $0$. The element $ab^{p-1}=\frac{a}{b} b^p$ is also a
  constant, so $ab^{p-1}=s(ab^{p-1})^p$. Hence \[ \frac{a}{b} =
  \Bigl(\frac{s(ab^{p-1})}{b}\Bigr)^p.\]
\end{proof}

From now on, we work inside a sufficiently saturated model $K$ of
$\dcfp$.
\begin{cor}\label{C:coh_DCFP}
 Consider two subfields $A$ and $B$ of $K$ containing an elementary
  substructure $M$ of $K$. Whenever \[A\ind^{\dcfp}_M B,\]
  \noindent the fields $K^p\cdot A$ and $K^p\cdot B$ are linearly
  disjoint over $K^p\cdot M$.
\end{cor}
\begin{proof}
  The quotient field $A'$ of the $\LL_{\delta,s}$-structure
  generated by $A$ is $K^p$-special, by the Remarks \ref{R:diffperf_sep} and
  \ref{R:closure_s}. The result now follows from
  Lemma \ref{L:coh}, as in the proof of Corollary \ref{C:coh_SCFinfty}.
\end{proof}
We will now present a relative quantifier elimination, by
isolating the formulae which will be our candidates for the
equationality of $\dcfp$.

\begin{definition}\label{D:delta-tame}
  Let $x$ be a tuple of variables. A formula $\phi(x)$ in the language
  $\LL_\delta$ is $\delta$\emph{-tame} if there are differential
  polynomials $q_1,\ldots, q_m$, with $q_i$ in the differential ring
  $\Z\{X,T_1,\ldots, T_{i-1}\}$, and a system of differential
  equations $\Sigma$ in $\Z\{X,T_1,\ldots, T_{n}\}$ such that
  \[
  \phi(x )\;\;=\;\;\exists\, z_1\ldots \exists z_n  \Bigl(
  \Land_{j=1}^{n} z_j^p\doteq q_j(x,z_1,\ldots, z_{j-1}) \land\;
  \Sigma(x,z_1,\ldots, z_n)\Bigr).\]
\end{definition}

\begin{prop}\label{P:delta_BK}
  Every formula in $\dcfp$ is a Boolean combination of $\delta$-tame
  formulae.
\end{prop}

\begin{proof}
  The proof is a direct adaptation of the proof of Proposition
  \ref{P:lambda_BK}. We need only show that the equation $t(x)\doteq0$
  is a Boolean combination of $\delta$-tame formulae, for every
  $\LL_{\delta,s}$-term $t(x)$. Proceed by induction on the number of
  occurrences of $s$ in $t$. Suppose that $t(x)=r(x,s(q(x)))$, for
  some $\LL_{\delta,s}$-term $r$ and a polynomial $q$, By induction,
  the equation $r(x,z)\doteq0$ is equivalent to a Boolean combination
  $BK(\psi_1(x,z),\dotsc)$ of $\delta$-tame formulae. Thus
  $t(x)\doteq0$ ist equivalent to
  \[\bigl(\neg\delta(q(x))\doteq0\;\land\;r(x,0)\doteq0\bigr)\;\,\lor\;\,
  \bigl(\delta(q(x))\doteq0\;\land\;
  BK(\exists z\,z^p\doteq q(x)\,\land\,\psi_1(x,z),\dotsc)\bigr),\] which is, by
  induction, a Boolean combination of $\delta$-tame formulae.
\end{proof}

We conclude this section with a homogenisation result for
$\delta$-tame formulae, as in Proposition \ref{P:lambda_hom}.

\begin{prop}\label{P:homo_dcf}
  Given a $\delta$-tame formula $\phi(x_1,\ldots, x_n)$ and natural numbers
  $k_1,\ldots, k_n$, there is a $\delta$-tame formula $\phi'(x_0,\ldots,
  x_n)$ such that
    \[  \dcfp\proves\forall x_0 \ldots \forall x_n \Bigl(
  \phi'(x_0,\ldots, x_n)\longleftrightarrow \Bigl(
  \phi\Bigl(\frac{x_1}{x_0^{k_1}},\ldots, \frac{x_n}{x_0^{k_n}}\Bigr) \lor x_0
  \doteq 0 \Bigr) \Bigr)  .\]
\end{prop}

\begin{proof}
We prove it by induction on the number of existential quantifiers iny
$\phi$. If $\phi$ is a system $\Sigma$ of differential
equations, rewrite

\[\Sigma(\frac{x_1}{x_0^{k_1}},\ldots, \frac{x_n}{x_0^{k_n}})
\Longleftrightarrow \frac{\Sigma'(x_0,\ldots, x_n)}{x^N_0} ,\]

\noindent for some natural number $N$ and a system of differential
equations $\Sigma'(x_0,\ldots, x_n)$. Set
\[ \phi'(x_0,\ldots, x_n)= x_0\cdot \Sigma'(x_0,\ldots, x_n) .\]

\noindent For a general $\delta$-tame formula, write
\[ \phi(x_1,\ldots, x_n)=\exists z\ \left( z^p\doteq q(x_1,\ldots,
    x_n)\land \psi(x_1,\ldots,x_n, z) \right) ,\]

\noindent for some polynomial $q$ and a $\delta$-tame formula $\psi$
with one existential quantifier less. There is a polynomial
$q'(x_0,\ldots, x_n)$ such that

\[ q(\frac{x_1}{x_0^{k_1}},\ldots,
  \frac{x_n}{x_0^{k_n}})=\frac{q'(x_0,\ldots,x_n)}{x_0^{pN-1}},\]

\noindent for some natural number $N$. By induction, there is a
$\delta$-tame formula $\psi'(x_0,\ldots, x_n, z)$ such that

\[ \dcfp\proves\forall x_0 \ldots \forall x_n \forall z
  \Bigl( \psi'(x_0,\ldots, x_n, z)\longleftrightarrow \Bigl(
  \psi\Bigl(\frac{x_1}{x_0^{k_1}},\ldots, \frac{x_n}{x_0^{k_n}},
  \frac{z}{x^N_0}\Bigr) \lor x_0 \doteq 0 \Bigr) \Bigr) .\]

\noindent Set now

\[ \phi'(x_0,\ldots, x_n) = \exists z \left( z^p\doteq x_0\cdot
    q'(x_1,\ldots, x_n)\land \psi'(x_0,x_1,\ldots,x_n, z) \right).\]
\end{proof}

\section{Equationality of $\dcfp$}\label{S:EqDCFp}

We have now all the ingredients to show that the theory $\dcfp$ of
existentially closed differential fields of positive characteristic
$p$ is equational. Working inside a sufficiently saturated model $K$
of $\dcfp$, given a $\delta$-tame formula in a fied partition of the
variables $x$ and $y$, one can show, similar to the proof of Theorem
\ref{T:eq_SCF}, that the set $\phi(a,y)$ is \icl. However, we will
provide a proof, which resonates with Srour's approach \cite{gS88},
using Lemma \ref{L:Srour}. We would like to express our gratitude to
Zo\'e Chatzidakis and Carol Wood for pointing out Srour's result.

\begin{theorem}[Srour \cite{gS88}]\label{T:eqDCFp}
  In any partition of the variables, the $\delta$-tame formula
  $\phi(x;y)$ is an equation.
\end{theorem}
\noindent Srour proved this for the equivalent notion of S-formulae,
cf.\  Definition \ref{D:S-formula} and Lemma \ref{L:S_and_delta_equiv}.
\begin{proof}
  We prove it by induction on the number $n$ of existential
  quantifiers. For $n=0$, the formula $\phi(x;y)$ is a system of
  differential equations, which is clearly an equation, by Lemma
  \ref{L:eq_allchar_delta}.

  For $n>0$, write $\phi(x,y)$ as

  \[\exists z \Bigl( z^p\doteq q(x,y)\; \land\; \psi(x,y,z) \Bigr),\]

  \noindent where $\psi(x,y,z)$ is a $\delta$-tame formula with $n-1$
  existential quantifiers.

  \begin{claimstern}
    Suppose that every differential monomial in $x$ occurs
    in $q(x,y)$ as a $p^\text{th}$-power. Then $\phi(x;y)$ is an equation.
  \end{claimstern}

  \begin{claimsternproof}
    It suffices to prove that $\phi(x,b)$ is equivalent to a
    $\delta$-tame formula $\psi'(x,b,b')$ with $n-1$ existential
    quantifiers, for some tuple $b'$. Choose a $K^p$-basis
    $1=b_0,\dotsc,b_N$ of the differential monomials in $b$
    occurring in $q(x,b)$ and write
    $q(x,b)=\sum_{i=0}^Nq_i(x,b')^p\cdot b_i$. Then
    $\phi(x,b)$ is equivalent  (in $\dcfp$) to
    \[\exists z \Bigl( z\doteq
    q_0(x,b')\;\land\; \Land\limits_{i=1}^N q_i(x,b')\doteq 0
    \land\; \psi(x,b,z) \Bigr),\] which is equivalent
    to \[\psi'(x,b,b')=\Bigr(\Land\limits_{i=1}^N q_i(x,b')\doteq
    0 \;\land\; \psi(x,b,q_0(x,b')) \Bigr).\]
  \end{claimsternproof}

  \noindent In order to show that $\phi(x;y)$ is an equation, we will
  apply Remark \ref{R:Srour}. Consider a tuple $a$ of
  length $|x|$, and two elementary substructures $M\subset N$ with
  $a\ind_M N$. Choose now a $K^p\cdot M$-basis $a_0,\dotsc,a_M$ of the
  differential monomials in $a$ which occur in $q(a,y)$ and write
  \[ q(a,y)=\sum\limits_{i=0}^N q_i(a'^p,m,y)\cdot a_i,\]
  for tuples $a'$ in $K$ and $m$ in $M$, and differential polynomials
  $q_i(x',y',y)$ with integer coefficients and linear in $x'$ and
  $y'$. Observe that we may assume that $a'\ind_{Ma} N$, which implies
  $aa'\ind_M N$.

  \noindent By Corollary \ref{C:coh_DCFP}, the elements
  $a_0,\dotsc,a_M$ remain linearly
  independent over $K^p\cdot N$. Thus, for all $b$ in $N$,
  \[ K\models \phi(a,b) \longleftrightarrow \psi'(a,a',m, b),\]
  where
  \[\psi'(x,x',y',y)\;\;=\;\;\exists z \Bigl( z^p\doteq
  q_0(x'^p,y',y) \land\; \Land\limits_{i=1}^N q_i(x'^p,y',y)\doteq 0
  \land\; \psi(x, y, z) \Bigr).\]
  By the previous claim, the $\delta$-tame formula $\psi'(x,x';y',y)$
  is an equation, so   \[ \tp(a,a'/M) \proves \tp_{\psi'}^+(a, a'/N).\]

\noindent   In order to show that
  \[ \tp(a/M) \proves \tp_{\phi}^+(a/N),\]
  consider a realisation $\tilde a$ of $\tp(a/M)$ and an instance
  $\phi(x,b)$ in $\tp_{\phi}^+(a/N)$.
  There is a tuple $\tilde a'$ such that $a a'\equiv_M \tilde a\tilde
  a'$. Since $K\models\psi'(a,a',m, b)$, we have $K\models \psi'(\tilde
  a,\tilde a',m,b)$.  Observe that there are $\tilde a_0,\dotsc,\tilde a_N$
  whith $q(\tilde a,y)=\sum\limits_{i=0}^N q_i(\tilde a'^p,m,y)\cdot
  \tilde a_i$, so we have
  in particular that
  \[ q(\tilde a, b) = q_0(\tilde a', m, b), \]
  whence $K\models \phi(\tilde a,b)$, as desired.

\end{proof}

Together with Proposition \ref{P:delta_BK}, we conclude the following result:

\begin{cor}\label{C:DCFPeq}
  The theory $\dcfp$ of existentially closed differential fields is
  equational.
\end{cor}

Similar to Corollary \ref{C:EI_SCF}, there is a partial elimination of
imaginaries for $\dcfp$, by Lemma \ref{L:eqEI} and Theorem
\ref{T:eqDCFp}. Unfortunately, we do not have either an explicit description
of the canonical parameters of instances of $\delta$-tame formulae.

\begin{cor}\label{C:EI_DCF}
  The theory $\dcfp$ of differentially closed fields of positive
  characteristic $p$ has weak elimination of imaginaries, after adding
  canonical parameters for all instances of $\delta$-tame formulae.
\end{cor}

\vspace{3em}


\subsection*{Digression: On Srour's proof of the  equationality of
  $\dcfp$}
\begin{definition}[Srour \cite{gS88}]\label{D:S-formula}
  An S-Formula $\phi$ is a conjunction of $\LL_{\delta,s}$-equations such
  that, for every subterm $s(r)$ of a term occurring in $\phi$, the
  equation $\delta(r)\doteq 0$ belongs to $\phi$.
\end{definition}

Srour's proof first shows that every formula is equivalent in $\dcfp$to a
Boolean combination of S-formulae. This follows from Proposition
\ref{P:delta_BK}, as the next Lemma shows.

\begin{lemma}\label{L:S_and_delta_equiv}
 Every tame $\delta$-formula is equivalent to an S-formula, and
 conversely.
\end{lemma}
\begin{proof}
  For every
  $\LL_{\delta,s}$-formula $\psi(x,z)$ and every polynomial
  $q(x)$, observe that
  \[\dcfp\;\proves\;\; \exists z\,\Bigl(z^p\doteq
  q(x)\land\psi\bigl(x,z\bigr)\Bigr)\;\longleftrightarrow\;\Bigl(\delta(q)\doteq
  0\land\psi\bigl(x,s(q(x))\bigr)\Bigr).\]
\end{proof}

In order to show that S-formulae $\phi(x;y)$ are equations, Srour uses
the fact that, whenever $A$ and $B$ are elementary submodels of a
model $K$ of $\dcfp$ which are linearly disjoint over their intersection $M$, then
for every $a\in A$ and any S-formula $\phi(x;y)$,
\[\tp(a/M)\proves\tp^+_\phi(a/B).\]
In order to do so, he observes that S-formulae are preserved under
differential ring homomorphisms, as well as a striking result of
Shelah (see the the proof of \cite[Theorem 9]{sS16}): the ring
generated by $A$ and $B$ is differentially perfect, that is, it is
closed under $s$. We would like to present a slightly simpler proof of
Shelah's result.

\begin{lemma}(Shelah's Lemma \cite[Theorem 9]{sS16})
  Let $M$ be a common differential subfield of the the two
  differential fields $A$ and $B$ and $R=A\otimes_MB$. If $M$ is
  existentially closed in $B$, then the ring of constants of $R$ is
  generated by $\C_A$ and $\C_B$.
\end{lemma}
\noindent In particular, in characteristic $p$, the ring $R$ is
 differentially perfect, whenever both $A$ and $B$ are.
\begin{proof}

  \begin{claim}
   The differential field $A$ is existentially closed in $R$. In
   particular $R$ is an integral domain.
  \end{claim}

  \begin{claimproof}
    Suppose $R\models\rho(a,r)$, for some  quantifier-free
    $\delta$-formula $\rho(x,y)$, and tuples $a$ in $A$ and $r$ in
    $R$. Rewriting $\rho$, we may assume that $r=b$ and $a$ occurs
    linearly in $\rho$ and is an enumeration of a basis of $A$ over
    $M$. In particular, there is a quantifier-free formula $\rho'(y)$
    such that  for all $b\in B$
\[ R\models \forall y \Bigl
      (\rho(a,b)\longleftrightarrow \rho'(b) \Bigr).\]
\noindent Since $M$ is existentially closed in $B$, and the validity
of quantifier-free formulae is preserved under substructures, we
conclude that there is some $a'$ in $M$ satisfying $\rho'(y)$ and thus
$\rho(a,a')$ holds in $R$, and hence in $A$.

  \end{claimproof}

  \noindent Let $K$ be the quotient field of $R$.

  \begin{claim}
    \label{Cl:prep}
    The ring of constants of the ring $R'$ generated by $A$ and $\C_B$
    is generated by $\C_A$ and $\C_B$.
  \end{claim}

  \begin{claimproof}
    Let $(a_i)$ be a basis of $A$ over $\C_A$, with $a_0=1$. Every
    $x$ in $R'$ can be written as $\sum_i a_i\cdot c_i$, for some $c_i$
    in the ring generated by $\C_A$ and $\C_B$. By Fact
    \ref{F:Wronski}, the $a_i$'s are independent over $\C_K$. If $x$
    is a constant in $R'\subset K$, then $x=c_0$.
  \end{claimproof}

  Fix now a basis $(a_i)_{i\in I}$ of $A$ over $M$ and let
  $x=\sum_{i\in I}a_i\cdot b_i$ be a constant in $R$.

  \begin{claim}
    All $\delta(b_j)$ are in the $M$-span of $\{b_i\mid i\in I\}$.
  \end{claim}

  \begin{claimproof}
    Write $\delta(a_j)=\sum_{i\in I}m_{j,i}a_i$ for $m_{j,i}\in M$.
    Since $0=\delta(x)$, we have
    \[0=\sum_ia_i\cdot\delta(b_i)+\sum_j(\sum_im_{j,i}a_i)\cdot b_j=
    \sum_ia_i\cdot\delta(b_i)+\sum_ia_i\cdot (\sum_jm_{i,j}b_j),\]
    whence $\delta(b_i)=-\sum_jm_{j,i}b_i$.
  \end{claimproof}

  \begin{claim}
    All $b_i$'s lie in the ring generated by $M$ and $\C_B$.
  \end{claim}

  \begin{claimproof}
    Let $b\in B^n$ be the column vector of the non-zero elements of
    $\{b_i\mid i\in I\}$. By the last claim, there is an $n\times
    n$-matrix $H$ with coefficients in $M$ with $\delta(b)=H\cdot b$. Let
    $u^1,\ldots, u^m$ be a maximal linearly independent system of
    solutions of $\delta(y)=H\cdot y$ in $M^n$. Consider the $n\times
    m$-matrix $U$ with columns $u^1,\ldots, u^m$. Since $M$ is
    existentially closed in $B$, the column vectors $u^1,\ldots u^m,
    b$ must be linearly dependent, so $b=U\cdot z$ for some vector
    $z$ in $B^m$. It
    follows that
    \[H\cdot b=\delta(b)=\delta(U)\cdot z+U\cdot \delta(z)=H\cdot U\cdot
      z+U\cdot \delta(z)=H\cdot b+ U\cdot \delta(z),\] whence $\delta(z)=0$.
  \end{claimproof}

  In particular, the constant $x$ lies in the ring $R'$ from Claim
  \ref{Cl:prep}, so  $x$ is in the ring generated by $\C_A$ and $\C_B$.
\end{proof}


\vspace{3em}


\subsection*{Interlude: An alternative proof of the equationality of
  $\scfpi$}

As a by-product of Theorem \ref{T:eqDCFp}, we obtain a different proof of the
equationality of $\scfpi$: We will show that every $\lambda$-tame
formula is an equation, since it is equivalent in a particular
model of $\scfpi$, namely a differentially closed field of
characteristic $p$, to a $\delta$-tame formula. A similar method will
appear again in Corollary \ref{C:eq_0}.

\begin{prop}\label{P:equiv_lambdadelta}
  Every $\lambda$-tame formula is equivalent in $\dcfp$ to a
  $\delta$-tame formula.
\end{prop}
\begin{proof}
  Work inside a model $(K, \delta)$ of $\dcfp$.
  The proof goes by induction on the degree of the $\lambda$-tame formula
  $\phi(x)$. If $\phi$ is a polynomial equation, there is nothing to
  prove. Since the result follows for conjunctions, we need only
  consider the particular
  case when $\phi$ is of the form:
  \[\phi(x)=\pdep_n(q_1,\dotsc, q_n)\;\lor\;
  (\ldef{q_0,\dotsc,q_n}\land\, \psi(x, \overline{\lambda}(q_0,
  \ldots, q_n))),\] for some $\lambda$-tame formula $\psi(x,
  z_1,\ldots,z_n)$ of strictly smaller degree and polynomials $q_0, \ldots, q_n$ in
  $\Z[x]$.

 \noindent Let $\W(x)=\W(q_1,\ldots,q_n)$ be the Wronskian of $q_1,\ldots,q_n$,
 that is, the determinant of the matrix
 \[
 A(x)=\begin{pmatrix}
 q_1 & q_2 & \ldots & q_n \\
 \delta(q_1) &  \delta(q_2) & \ldots &  \delta(q_n) \\
 \vdots & & \vdots & \\
  \delta^{n-1}(q_1) &  \delta^{n-1}(q_2) & \ldots &
  \delta^{n-1}(q_n)
 \end{pmatrix}.\]

 \noindent  and $B(x)$ be the adjoint matrix of  $A(x)$. Set
 \[ D(x)= \begin{pmatrix} q_0 \\ \delta(q_0)\\ \vdots\\
   \delta^{n-1}(q_0) \end{pmatrix}.\]

\noindent Since $K$ is differentially perfect, the elements
$q_1(x),\ldots,q_n(x)$ are linearly independent over $K^p$ if and only
if $\W(x)\neq 0$. In that case, the functions
$\overline{\lambda}(q_0,\dotsc,q_n)$ are defined if and only if every
coordinate of the vector ${\W(x)}\inv\!\cdot B(x)\cdot D(x)$ is a
constant, in which case we have
 \[ \overline\lambda(q_0,\dotsc,q_n)^p={\W(x)}\inv\!\cdot B(x)\cdot D(x),\]
 or equivalently,
 \[ (\W(x)\cdot \overline\lambda(q_0,\dotsc,q_n))^p=
   \W(x)^{p-1}\!\cdot B(x)\cdot D(x) \]

 \noindent By induction, the formula $\psi(x,z_1,\ldots, z_n)$ is
 equivalent to a $\delta$-tame formula $\psi_\delta$. Homogenising
 with respecto to $z_0, z_1,\dotsc, z_n$, as in Proposition
 \ref{P:homo_dcf}, there is a $\delta$-tame formula
 $\psi'_\delta(x, z_0,z_1,\ldots, z_n)$ equivalent
 to
 \[\psi_\delta(x,\frac{z_1}{z_0},\ldots, \frac{z_n}{z_0} )\lor
   z_0\doteq 0\] Therefore, if $z=(z_1,\cdots,z_n)$, then
 \[K\models \Bigl(\phi(x) \longleftrightarrow \exists z\;
 \bigl(z^p\doteq \W(x)^{p-1}\cdot B(x)\cdot
 D(x)\;\land\;\psi'_\delta(x,\W(x),z)\bigr) \Bigr).\]
 The right-hand side is a $\delta$-tame formula, as desired.
\end{proof}

By Propositions \ref{P:lambda_BK} and \ref{P:equiv_lambdadelta}, and Theorem
\ref{T:eqDCFp}, we obtain a different proof of Corollary \ref{C:SCF_eq}:

\begin{cor}
The theory $\scfpi$ of separably closed fields of characteristic
  $p>0$ and infinite imperfection degree is equational.
\end{cor}


\section{Model Theory of Pairs}\label{S:MTPairs}

The last theory of fields we will consider in this work is the
(incomplete) theory $\acfp$ of proper pairs of algebraically closed
fields. Most of the results mentioned here appear in \cite{jK64, Po83,
  BPV03}.

Work inside a sufficiently saturated model $(K,E)$ of $\acfp$ in the
language $\LL_P=\LL_{rings}\cup\{P\}$, where $E=P(K)$ is the
 proper subfield. We will use the index $P$ to refer to the
expansion $\acfp$.

\noindent A subfield $A$ of $K$ is \emph{tame} if $A$ is algebraically
independent from $E$ over $E_A=E\cap A$, that is, \[
  A\ind^\mathsf{ACF}_{E_A} E.\]
\noindent Tameness was called \emph{$P$-independence} in
\cite{BPV03}, but in order to avoid a possible confusion, we have
decided to use a different terminology.

\begin{fact}\label{F:Kiesler}
  The completions of the theory $\acfp$ of proper pairs of algebraically
  closed fields are obtained once the characteristic is fixed. Each of
  these completions is $\omega$-stable of Morley rank $\omega$. The
  $\LL_P$-type of a tame subfield of $K$ is uniquely determined by its
  $\LL_P$-quantifier-free type.

\noindent Every subfield of $E$ is automatically tame, so the
induced structure on $E$ agrees with the field structure. The subfield
$E$ is a pure algebraically closed field and has Morley rank $1$.

\noindent If $A$ is a tame subfield, then its $\LL_P$-definable closure
coincides with the inseparable closure of  $A$ and its
$\LL_P$-algebraic closure is the field algebraic closure $\acl(A)$ of
$A$, and $E_{\acl_P(A)}=\acl(E_A)$.
\end{fact}

Based on the above fact, Delon \cite{fD12} considered the following
expansion of the language $\LL_P$:

\[\LLD=\LL_P\cup\{\dep_n,
\lambda_n^i\}_{1\leq i \leq n\in \N},\] where the relation $\dep_n$ is
defined as follows:
\[K\models \dep_n(a_1,\ldots,a_n) \Longleftrightarrow a_1,\ldots,a_n
\text{ are $E$-linearly independent},\]
and the $\lambda$-functions take values in $E$ and are defined by the
equation
\[ a_0 = \sum\limits_{i=1}^n \lambda_n^i(a_0, a_1\ldots,a_n) \, a_i,\]
if
$K\models \dep_n(a_1,\ldots,a_n) \land \neg
\dep_{n+1}(a_0,a_1,\ldots,a_n)$, and are $0$ otherwise. Clearly, a
field $A$ is closed under the $\lambda$-functions if and only if it is
linearly disjoint from $E$ over $E_A$, that is, if it is $P$-special,
as in Definition \ref{D:spec}. Note that the fraction field of an
$\LLD$-substructure is again closed under $\lambda$-functions and thus
it is tame. The theory $\acfp$ has therefore quantifier elimination
\cite{fD12} in the language $\LLD$. Note that the formula $P(x)$ is
equivalent to $\dep_2(1,x)$. Likewise, the predicate $\dep_n$ is is
equivalent to $\lambda_n^1(a_1, a_1\ldots,a_n)=1$.

Since the definable closure of a set is $P$-special, we conclude
the following result by Lemma \ref{L:coh}.

\begin{cor}\label{C:coh_pair}
  Given two subfields $A$ and $B$ of $K$ containing an
  $\LL_p$-elementary substructure $M$ of $K$ such that $A\ind^P_M B$,
  then the fields $E \cdot A$ and $E\cdot B$ are linearly disjoint
  over $E\cdot M$.
\end{cor}

Our candidates for the equations in the theory $\acfp$ will be called
\emph{tame} formulae.

\begin{definition}\label{D:tame_formel}
  Let $x$ be a tuple of variables. A formula $\phi(x)$ in the
  language $\LL_P$ is \emph{tame} if there are polynomials
  $q_1,\ldots, q_m$ in $\Z[X,Z]$, homogeneous in the variables $Z$,
  such that
  \[
  \phi(x )\;\;=\;\;\exists\, \zeta \in P^r \biggl(\neg\zeta\doteq 0\;
  \land\; \Land_{j\leq m} q_j(x,\zeta)\doteq 0 \biggr).\]
\end{definition}

\begin{lemma}\label{L:segre}
  Let $q_1,\ldots, q_m\in\Z[X,Y,Z]$ be polynomials, homogeneous in the
  variables $Y$ and $Z$ separately. The $\LL_P$-formula
  \[\exists\, \upsilon \in P^r\; \exists \zeta \in
  P^s\Bigl(\neg\upsilon\doteq 0\; \land\; \neg\zeta\doteq
  0\;\land\;\Land_{k\leq m} q_k(x,\upsilon,\zeta)\doteq 0 \Bigr)\]
  is equivalent in $\acfp$ to a tame formula.
\end{lemma}
\begin{proof}
  With the notation $\xi_{\ast,j}=\xi_{1,j},\dotsc,\xi_{r,j}$ and
  $\xi_{i,\ast}=\xi_{i,1},\dotsc,\xi_{i,s}$, the previous formula is
  equivalent in $\acfp$   to the tame formula
  \[\exists (\xi_{1,1},\dotsc,\xi_{rs})\in P^{r,s}\setminus 0
  \Land_{i,j,k=1}^{r,s,m}q_k(x,\xi_{*,j},\xi_{i,*})\doteq0.\]
\end{proof}

\begin{cor}\label{C:conj_tame}
  The collection of tame formulae is closed under conjunctions and
  disjunctions.
\end{cor}

In order to prove that tame formulae determine the type in $\acfp$, we
need a short observation regarding the $E$-annihilator of a (possibly
infinite) tuple. Fix some enumeration
$\left(M_i(x_1,\dotsc,x_s)\right)_{i=1,2,\ldots}$ of all monomials in
$s$ variables. Given a tuple $a$ of length $s$, denote
\[\Ann_n(a) =\biggl\{(\lambda_1,\ldots,\lambda_n) \in \E^n\;
\biggm | \;\sum\limits_{i=1}^n \lambda_i\cdot M_i(a) =0\biggr \}.\]

\begin{notation}
  If we denote by $x\cdot y$ the scalar multiplication of two tuples
  $x$ and $y$ of length $n$, that is
  \[ x\cdot y= \sum\limits_{i=1}^n x_i\cdot y_i,\] then
\[ \Ann_n(a) =\{ \lambda \in \E^n \mid  \lambda \cdot (M_1(a),\ldots,
M_n(a))=0\}.\]
\end{notation}

\begin{lemma}\label{L:tp_ann}
  Two tuples $a$ and $b$ of $K$ have the same type if and only if
  \[\ldim_E \Ann_n(a) =\ldim_E \Ann_n(b)\] \noindent and the type
  $\tp(\pk(\Ann_n(a)))$ equals $\tp(\pk(\Ann_n(a)))$ (in the pure
  field language), for every $n$ in $\N$.
\end{lemma}
\begin{proof}
  We need only prove the right-to-left implication. Since $\pk(\Ann_i(a))$ is
  determined by $\pk(\Ann_n(a))$, for $i\leq n$, we obtain an
  automorphism of $E$ mapping $\pk(\Ann_n(a))$ to $\pk(\Ann_n(b))$ for
  all $n$. This automorphism maps $\Ann_n(a)$ to $\Ann_n(b)$ for all
  $n$ and hence extends to an isomorphism of the rings $E[a]$ and
  $E[b]$. It clearly extends to a field isomorphism of the tame
  subfields $E(a)$ and $E(b)$ of $K$, which in turn can be extended to
  an automorphism of $(K,E)$. So $a$ and $b$ have the same
  $\acfp$-type, as required.
\end{proof}

\begin{prop}\label{P:Tame_Type}
  Two tuples $a$ and $b$ of $K$ have the same $\acfp$-type if and
  only if they satisfy the same tame formulae.
\end{prop}

\begin{proof}
Let $q_1(Z),\ldots, q_m(Z)$ be homogeneous polynomials over $\Z$.
By Lemma \ref{L:tp_ann}, it suffices to show that

\begin{center}
  \guillemotleft \, $\Ann_n(x)$ has a $k$-dimensional subspace $V$
  such that $\Land_{j\leq m} q_j(\pk(V))=0 $ \guillemotright
\end{center}
is expressible by a tame
formula. Indeed, it suffices to guarantee that there is an
element $\zeta$ in $\Gr_k(E^n)$ such that \[(e\pluck \zeta)\cdot
(M_1(x),\ldots, M_n(x))=0\]
\noindent for all $e$ from a a fixed basis of $\extp^{k-1} (E^n)^*$,
and
\[ \Land_{j\leq m} q_j(\zeta)=0.\]
In particular, the tuple $\zeta$ is not trivial, so we conclude that the
above is a tame formula.
\end{proof}

By compactness, we conclude the following:

\begin{cor}\label{C:Tame_BK}
  In the (incomplete) theory $\acfp$ of proper pairs of algebraically
  closed fields, every formula is a Boolean combination of tame
  formulae.
\end{cor}

\section{Equationality of belles paires of algebraically closed
  fields}\label{S:EqPairs}

In order to show that the stable theory $\acfp$ of proper pairs of
algebraically closed fields is equational, we need only consider tame
formulae with respect to some partition of the variables, by Corollary
\ref{C:Tame_BK}. As before, work inside a sufficiently saturated model
$(K,E)$ of $\acfp$ in the language $\LL_P=\LL_{rings}\cup\{P\}$, where
$E=P(K)$ is the proper subfield.

\noindent Consider the following special case as an auxiliary result.

\begin{lemma}\label{L:tame_vble_in_P}
  Let $\phi(x;y)$ be a tame formula. The formula
  \[\phi(x;y) \land x \in P\] is an equation.

\end{lemma}

\begin{proof}

  Let $b$ be a tuple in $K$ of length $|y|$, and suppose that the
  formula $\phi(x,b)$ has the form

  \[
\phi(x,b)\;\;=\;\;\exists\, \zeta \in P^r \biggl(\neg\zeta\doteq 0\,
\land\; \Land_{j\leq m} q_j(x,b,\zeta)\doteq 0 \biggr).\]

\noindent for some polynomials $q_1, \ldots,q_m$ with integer
coefficients and  homogeneous in  $\zeta$. Express each of
the monomials in $b$ appearing in the above equation as a linear
combination of a basis of $K$ over $E$. We see that there are
 polynomials $r_1, \ldots, r_s$ with coefficients in $E$,
homogeneous in  $\zeta$, such that the formula
$\phi(x,b) \wedge x\in P$ is equivalent to
 \[\exists\, \zeta \in P^r\;
\biggl(\neg\zeta\doteq 0\; \land\; \Land_{j\leq s} r_j(x,\zeta)\doteq
0 \biggr).\]

\noindent Working inside the algebraically closed subfield $E$, the
expression inside the brackets is a projective variety, which is hence
complete. By Remark \ref{R:proj_complete}, its projection is again
Zariski-closed, as desired.
\end{proof}

\begin{prop}\label{P:tame_ic}
 Let $\phi(x;y)$ be a tame formula.  The formula $\phi(x;y)$ is an equation.
\end{prop}

\begin{proof}

We need only show that every instance $\phi(a,y)$ of a tame formula
is \icl. By Lemma \ref{L:icl_model}, it suffices to consider a Morley
sequence $(b_i)_{i\leq \omega}$ over an elementary substructure $M$ of
$(K,E)$ with \[ a\ind^P_M b_i \text{ with } \models \phi(a,b_i)
\text{ for } i <\omega.\]

\noindent Suppose that the formula $\phi(a,y)$ has the form

\[\phi(a,y)\;\;=\;\;\exists\, \zeta \in P^r \biggl(\neg\zeta\doteq
0\;
\land\; \Land_{j\leq m} q_j(a,y,\zeta)\doteq 0 \biggr),\]
\noindent for polynomials $q_1,\ldots,q_n$ with integer coefficients
and homogeneous in $\zeta$.

\noindent By Corollary \ref{C:coh_pair}, the fields $E\cdot M(a)$ and
$E\cdot M(b_i)$ are linearly disjoint over $E(M)$ for every
$i<\omega$. A basis $(c_{\nu})$ of $E \cdot M(a)$ over $E \cdot M$
remains thus linearly independent over $E\cdot M(b_i)$. By
appropriately writing each monomial in $a$ in terms of the basis
$(c_{\nu})$, and after multiplication with a common denominator, we
have that
\[\phi(a,y)\;\;=\;\;\exists\, \zeta \in P^r \biggl(\neg\zeta\doteq
0\;\land\; \Land_{\nu}
  r_\nu(e,m,y,\zeta)\cdot c_\nu \doteq 0 \biggr),\]
\noindent where $e$ a tuple from $E$ and $m$ is a tuple from $M$, and
the polynomials $r_\nu(X,Y',Y,Z)$ are homogeneous in $Z$. Hence,
linearly disjointness implies that
\[K\models \;\;\exists\, \zeta \in P^r \biggl(\neg\zeta\doteq
0\;\land\; \Land_{\nu}
  r_\nu(e,m,b_i,\zeta)\doteq 0 \biggr) \text{ for }
i<\omega.\]
By Lemma \ref{L:tame_vble_in_P}, the formula
\[\phi'(e,y',y)=\exists\, \zeta \in P^r \biggl(\neg\zeta\doteq
  0\;\land\; \Land_{\nu} r_\nu(e,y',y,\zeta)\doteq 0 \biggr)\] is
indiscernibly closed. Since the sequence $(m,b_i)_{i\leq\omega}$ is
indiscernible, we have $K\models\phi'(e,m,b_\omega)$, so
$K\models\phi(a,b_\omega)$, as desired.

\end{proof}

\noindent Corollary \ref{C:Tame_BK} and Proposition \ref{P:tame_ic}
yield now the equationality of $\acfp$.

\begin{theorem}\label{T:eq}
The theory of proper pairs of algebraically closed fields of a
fixed characteristic is equational.
\end{theorem}
\vspace{3em}


\subsection*{Encore! An alternative proof of equationality for
  pairs in characteristic $0$}

We will exhibit an alternative proof to the equationality of the
theory $T_p$ of belles paires of algebraically closed fields in
characteristic $0$, by means of differential algebra,  based on
an idea of G\"unayd\i n \cite{aG16}.

\begin{definition}\label{D:Ehull}
  Consider an arbitrary field $K$ with a subfield $E$. A subspace of
  the vector space $K^n$ is $E$\emph{-defined} if it is generated by
  vectors from $E^n$.

 \noindent  Since the intersection of two $E$-defined subspaces is again
  $E$-defined, every subset $A$ of $K$ is contained in a smallest
  $E$-defined subspace $A^E$, which we call the \emph{$E$-hull} of
  $A$.
\end{definition}
\begin{notation}
We write $v^E$ to denote $\{v\}^E$. Clearly $A^E$ is the sum of all
$v^E$ for $v$ in $A$. The $E$-hull $v^E$ can be computed as follows:
Fix a basis $(c_\nu\mid\nu\in N)$ of $K$ over $E$ and write
$v=\sum_{\nu\in N}c_\nu e_\nu$ for vectors $e_\nu\in E^n$. Then
$\{e_\nu\mid\nu\in N\}$ is a generating set of $v^E$.
\end{notation}
\noindent Similarly, every subset $A$ of the ring of polynomials
$K[X_1,\dotsc,X_n]$ has an $E$-hull $A^E$, that is, the smallest
$E$-defined subspace of $K[X_1,\dotsc,X_n]$.

\begin{lemma}\label{L:Ideal_E}
  Let $I$ be an ideal of $K[X_1,\dotsc,X_n]$. Then $I^E$ is the
  smallest ideal containing $I$ and generated by elements of
  $E[X_1,\dotsc,X_n]$.
\end{lemma}

\begin{proof}
  An ideal $J$ generated by the polynomials $f_i$ in
  $E[X_1,\dotsc,X_n]$ is generated, as a vector space, by the products
  $X_jf_i$. Conversely, for each variable $X_j$, the vector space
  $\{f\mid X_jf\in I^E\}$ is $E$-defined and contains $I$. Thus it
  contains $I^E$, so $I^E$ is an ideal.
\end{proof}

\noindent If the ideal $I$ is generated by polynomials $f_i$, then the
union of all $f_i^E$ generates the ideal $I^E$. Note also that, if $I$
is \emph{homogeneous}, i.e.\ it is the sum of all
\[I_d=\{f\in I\mid h \text{ homogeneous of degree } d\},\] then so is
$I^E$, with
$(I^E)_d=(I_d)^E$.

From now on, consider a sufficiently saturated algebraically
closed differential field $(K, \delta)$, equipped with a non-trivial
derivation $\delta$. Denote its field of constants $\C_K$ by
$E$. For example, we may choose $(K,\delta)$ to be a saturated
model of $\mathsf{DCF}_0$, the elementary theory of differential
closed fields of characteristic zero.

\noindent Observe that the pair $(K,E)$ is a model of the theory
$\acfp_0$ of proper extensions of algebraically closed fields in
characteristic $0$. In order to show that this theory is equational,
it suffices to show, by Proposition \ref{P:Tame_Type}, that every
instance of a tame formula determines a Kolchin-closed set in
$(K,\delta)$. We first need some auxiliary lemmata on the differential
ideal associated to a system of polynomial equations.

\begin{lemma}
  Let $v$ be a vector in $K^n$. Then the $E$-hull of $v$ is generated
  by $v,\delta(v),\ldots,\delta^{n-1}(v)$.
\end{lemma}

\begin{proof}
  Any $E$-defined subspace is clearly closed under $\delta$. Thus, we
  need only show the the subspace $V$ generated by
  $v,\delta(v),\ldots,\delta^{n-1}(v)$ is $E$-defined. Let $k\leq n$
  be minimal such that $v$ can be written as
  \[ v = a_1 e_1+\dotsb+a_ke_k\] for some elements $a_i$ in $K$ and
  vectors $e_i$ in $E^n$. Thus, the $e_i$'s are linearly independent
  and generate $v^E$. Hence $V \subset v^E$. If the dimension of $V$
  is strictly smaller than $k$, then
  $v,\delta(v),\ldots,\delta^{k-1}(v)$ are linearly dependent over
  $K$. The rows of the matrix
  \[
  \begin{pmatrix}
    a_1 & a_2 & \ldots & a_k \\
    \delta(a_1) &  \delta(a_2) & \ldots &  \delta(a_k) \\
    \vdots & & \vdots & \\
    \delta^{k-1}(a_1) &  \delta^{k-1}(a_2) & \ldots &
    \delta^{k-1}(a_k)
  \end{pmatrix} \]
  are thus linearly dependent over $K$. It follows from Fact
  \ref{F:Wronski} that $a_1,\ldots,a_k$ are linearly dependent over
  $E$. So there are $\xi_i$ in $E$, not all zero, such that
  $\xi_1a_1+\dotsb+\xi_ka_k=0$. The vector space

\[\Bigl\{ \sum_{i=1}^kb_ie_i \Bigm| \sum_{i=1}^k\xi_ib_i=0\Bigr\},\]
which contains $v$, has a basis from  $E^n$ and dimension strictly
smaller than $k$, contradicting the choice of the $e_i$'s.
\end{proof}

In order to apply the previous result, consider the derivation $D$ on
the polynomial ring $\K[X_1,\ldots, X_n]$ obtained by differentiating
the coefficients of a polynomial in $K$ (and setting $D(X_i)=0$, for
$1\leq i\leq n$). We say that an ideal $I$ of $\K[X_1,\ldots, X_n]$ is
\emph{differential} if it is closed under $D$.

\begin{cor}\label{C:diff_ideal}
  An ideal of $\K[X]$ is differential if and only if it can be
  generated by elements from $\E[X]$.
\end{cor}

\begin{cor}\label{C:bounds_hom}
  Given homogeneous polynomials $h_0,\dots,h_m$ in $\K[X]$ of a fixed
  degree $d$, there exists an integer $k$ in $\N$ (bounded only in
  terms of $d$ and the length of $X$) such that the ideal generated by
  $\{D^{j}(h_i)\}_{\substack{i\leq m \\ j < k}}$ is a differential and
  homogeneous ideal.
\end{cor}

We now have all the ingredients in order to show that tame formulae
are equations.

\begin{prop}\label{P:Kolchin}
  Let $\phi(x,y)$ be a tame formula. The definable set $\phi(x,b)$ is
  Kolchin-closed set in $(K,\delta)$.
\end{prop}

\begin{proof}

Suppose that

\[\phi(x,b)\;\;=\;\;\exists\, \zeta \in P^r \biggl(\neg\zeta\doteq
0\;
\land\; \Land_{i\leq m} q_i(x,\zeta)\doteq 0 \biggr),\]
\noindent for polynomials $q_j(X,Z)$ over $K$ homogeneous in $Z$ of
some fixed degree $d$.  Let $k$ be as in Corollary \ref{C:bounds_hom}.

\noindent For a tuple  $a$ in $K$ of length $|x|$, write

\[
    D^j(q_i(a,Z))= q_{i,j}(a,\dots,\delta^j(a),Z),\]

  \noindent for polynomials $q_{i,j}(X_0,\ldots, X_k, Z)$ over $K$,
  homogeneous in $Z$. By Corollary \ref{C:bounds_hom}, the ideal
  $I(a,Z)$ generated by
 \[\{q_{i,j}(a,\dots,\delta^j(a), Z)\}_{\substack{i< m \\ j<k }} \]

\noindent has a generating set consisting of
homogeneous polynomials

\[g_1(Z),\ldots,g_s(Z)\]

\noindent with coefficients in $\E[Z]$.

Now, since $\zeta$ ranges over the constant field,  the tuple $a$
realises $\phi(x,b)$ if and only if

\[(K,E)\models\;\;\exists\, \zeta \in P^r \biggl(\neg\zeta\doteq
0\;
\land\; I(a,\zeta)\doteq 0 \biggr),\]

\noindent  which is equivalent to

\[(K,E)\models\;\;\exists\, \zeta \in P^r \biggl(\neg\zeta\doteq
0\;
\land\; \Land_{i\leq s}g_i(\zeta)\doteq 0 \biggr),\]

\noindent The field $E$ is an elementary substructure of $K$, so the
above is equivalent to
\[K\models\;\;\exists\,\zeta\;\biggl(\neg\zeta\doteq
0\;
\land\; \Land_{i<s}g_i(\zeta)\doteq 0 \biggr),\]
which is again equivalent to
\[K\models\;\;\exists\, \zeta \biggl(\neg\zeta\doteq
0\;
\land\; I(a,\zeta)\doteq 0 \biggr).\]

\noindent Since $I(a,Z)$ is homogeneous, the Zariski-closed set it
determines is complete, hence its projection is given by a finite
number of equations $X(a,\ldots,\delta^{k-1}(a))$. Thus, the tuple $a$
realises $\phi(x,b)$ holds if and only if

\[ (K,\delta)\models X(a,\dots,\delta^{k-1}(a)),\]

\noindent which clearly describe a Kolchin-closed set, as desired.
\end{proof}

By Corollary \ref{C:Tame_BK}, we conclude the following:

\begin{cor}\label{C:eq_0}
  The theory $\acfp_0$ of proper pairs of algebraically closed fields
  of characteristic $0$ is equational.
\end{cor}
A definable set $\{a\in K^n\mid(K,E)\models\phi(a,b)\}$ is
\emph{t-tame}, if $\phi$ is tame, for some $b$ a tuple in $K$.
\begin{cor}
  In models of $\acfp_0$, the family of t-tame sets has the DCC.
\end{cor}
\begin{proof}
  The Kolchin topology is noetherian, by Ritt-Raudenbush's Theorem.
\end{proof}
\begin{question}
  Do  t-tame sets have the DCC in arbitrary characteristic?
\end{question}


\section{Appendix: Linear Formulae}\label{S:LinFor}

A stronger relative quantifier elimination was provided in
\cite[Theorem 1.1]{aG16}, , which yields a nicer description of the
equations to consider in the theory $\acfp_0$. We will provide an
alternative approach to G\"unayd\i n's result, valid in arbitrary
characteristic. We work inside a sufficiently saturated model $(K,E)$
of $\acfp$.

A tame formula $\phi(x)$ (cf. Definition \ref{D:tame_formel})
is \emph{linear} if the corresponding polynomials in $\phi$ are
linear in $Z$, that is, if there is a matrix $(q_{i,j}(X))$ of
polynomials with integer coefficients such that

\[\phi(x)\;\;=\;\;\exists \zeta\in P^s\; \left(\neg\zeta\doteq
  0\; \land\;
\Land_{j=1}^k \zeta_1q_{1,j}(x)+\dotsb+\zeta_sq_{s,j}(x)\doteq0\right).\]

\noindent A linear formula is \emph{simple} if $k=1$, that is, if it has
the form
\[\dep_s(q_1(x),\dotsc,q_s(x)),\]
 for polynomials $q_i$ in $Z[X_1,\ldots, X_n]$.

 We will show that every tame formula is equivalent in $\acfp$ to a
 conjunction of simple linear formulae. We first start with an easy
 observation.

\begin{lemma}\label{L:linear}
  Every tame formula is equivalent in $\acfp$ to a linear tame formula.
\end{lemma}
\begin{proof}
  Consider a tame formula
  \[
  \phi(x )\;\;=\;\;\exists\, \zeta \in P^r \biggl(\neg\zeta\doteq
  0\; \land\; \Land_{j\leq m} q_j(x,\zeta)\doteq 0 \biggr).\]

  \noindent
  Denote by $Z$ the tuple of variables
  $(Z_1,\ldots, Z_{\text{length}(\zeta)})$. For a tuple $a$ in $K$ of
  length $|x|$, denote by $I(a,Z)$ the ideal in $K[Z]$ generated by
  $q_1(a,Z),\dotsc q_m(a,Z)$. Recall the definition of the $E$-hull
  $I(a,Z)^E$ of $I(a,Z)$ (Definition \ref{D:Ehull}). Since
  $I(a,Z)\subset I(a,Z)^E$, a zero of $I(a,Z)^E$ is a zero of
  $I(a,Z)$. A relative converse holds: If the tuple
  $\zeta$ in $E^r$ is a zero of the ideal $I(a,Z)$, then $I(a,Z)$ is contained
  in the ideal generated by all $Z_i-\zeta_i$'s, which is $E$-defined,
  so $\zeta$ is a zero of $I(a,Z)^E$. As in the proof of Proposition
  \ref{P:Kolchin}, we conclude that $(K,E)\models\phi(a)$ if an only
  if $I^E(a,Z)$ has a non-trivial zero in $K^r$.

  \noindent The ideal $I(a,Z)^E$ is generated by polynomials from
  $q_j(a,Z)^E$.  In particular, there is a degree $d$, independent
  from $a$, such that $I^E(a,Z)$ has a non-trivial zero if and only if
  the $E$-hull $(I(a,Z)^E)_d$ of $I(a,Z)_d$ is not all of $K[Z]_d$. As
  a vector space, the ideal $I(a,Z)_d$ is generated by all products
  $M\cdot q_j(a,Z)$, with $M$ a monomial in $Z$ such that
  $\deg(M)+\deg_Z(q_j(X,Z))=d$. Given an enumeration $M_1,\dotsc,M_s$
  of all monomials in $Z$ of degree $d$, the vector space $I(a,Z)_d$
  is generated by a sequence of polynomials $f_1,\dotsc,f_k$ of the
  form
  \[f_j=M_1r_{1,j}(a)+\dotsb+M_sr_{s,j}(a),\] for
  polynomials $r_{i,j}(X)\in\Z[X]$ which do not depend of $a$. Thus,
  the tuple $a$ realises $\phi(x)$ if and only if
  $(I(a,Z)^E)_d\not=K[Z]_d$, that is, if and only if there is a tuple
  $\xi\in E^s\setminus 0$ such that $
  \xi_1r_{1,j}(a)+\dotsb+\xi_sr_{s,j}(a)=0$ for all
  $j=1,\dotsc,k$. The latter is expressible by a linear formula.
\end{proof}

In order to show that every tame formula is equivalent to a
conjunction of simple linear formulae, we need the following result:

\begin{prop}\label{P:simple}
  For all natural numbers $m$ and $n$, there is a natural number $N$
  and an $n\times N$-matrix
  $(r_{j,k})$ of polynomials from $\Z[x_{1,1},\dotsc,x_{m,n}]$ such
  that the linear formula
  \begin{equation}\label{linpaar}
    \;\;\exists\, \zeta \in P^r \left(\neg\zeta\doteq
  0\; \land\;  \Land_{j=1}^n
    \zeta_1x_{1,j}+\dotsb+\zeta_mx_{m,j}\doteq0 \right).
  \end{equation}
  is equivalent in $\mathrm{\acfp}$  to the conjunction of
  \begin{align}\label{minoren}
    &\Land_{j_1<\dotsb<j_m}\det((x_{i,j_{i'}}))\doteq
    0\\
      \intertext{and}
      \label{reduktion}
      &\Land_{k=1}^N\dep_m\bigl(\sum_{j=1}^nx_{1,j}r_{j,k}(\bar
      x),\dotsc,\sum_{j=1}^nx_{m,j}r_{j,k}(\bar x)\bigr).
  \end{align}
\end{prop}
\begin{proof}
  The implication
  $\eqref{linpaar} \Rightarrow \left( \eqref{minoren} \land
  \eqref{reduktion} \right)$ always holds, regardless
  of the choice of the polynomials $r_{j,k}$: Whenever a
  matrix $A=(a_{i,k})$ over $K$ is such that there is a non-trivial
  vector  $\zeta$ in  $E^m$ with \[ \Land_{j=1}^n
  \sum_{i=1}^m\zeta_ia_{i,j}=0,\]

  \noindent then the  rows of $A$ are linearly dependent, so
  $\det((a_{i,j_{i'}}))=0$ for all $j_1<\dotsb<j_m$. For all $k$,
  we have that
  \[\sum_{i=1}^m\zeta_i\bigl(\sum_{j=1}^na_{i,j}r_{j,k}(\bar a)\bigr)=
  \sum_{j=1}^n\bigl(\sum_{i=1}^m\zeta_ia_{i,j}\bigr)r_{j,k}(\bar a)=0.\]

\noindent For the converse, an easy compactness argument yields the
existence of the polynomials $r_{j,k}$, once we show that
$\eqref{linpaar}$ follows from $\eqref{minoren} $ together with the
infinite conjunction
  \begin{equation}\label{infreduktion}
    \Land_{r_1,\dotsc r_n\in\Z[\bar
        x]}\dep_m\bigl(\sum_{j=1}^nx_{1,j}r_j(\bar
    x),\dotsc,\sum_{j=1}^nx_{m,j}r_j(\bar x)\bigr).
  \end{equation}

  \noindent Hence, let $A=(a_{i,k})$ be a matrix over $K$ witnessing
  \eqref{minoren} and \eqref{infreduktion}. The rows of $A$ are
  $K$-linearly dependent, by \eqref{minoren}. If the matrix were
  defined over $E$, its rows would then be $E$-linearly dependent,
  which yields \eqref{linpaar}. Thus, if we $R$ is the subring of $K$
  generated by the entries of $A$, we may assume that the ring
  extension $E\subset E[R]$ is proper.

\begin{claim}\label{C:nonunit} There is a non-zero element $r$ in $R$
  which  is not a unit in $E[R]$.
\end{claim}

\begin{claimproof} The field $E(R)$ has transcendence degree
  $\tau\geq 1$ over $E$. As in the proof of Noether's Normalisation
  Theorem \cite[Theorem X 4.1]{sL84}, there is a transcendence basis
  $r_1,\dotsc,r_\tau$ of $R$ over $E$, such that $E[R]$ is an integral
  extension of $E[r_1,\dotsc,r_\tau]$. If $r_1$ were a unit in $E[R]$,
  its inverse would $u$ be a root of a polynomial with coefficients in
  $E[r_1,\dotsc,r_\tau]$ and leading coefficient $1$. Multiplying by a
  suitable power of $r_1$, we obtain a non-trivial polynomial relation
  among the $r_j's$, which is a contradiction.
\end{claimproof}

\begin{claim}\label{C:seq_ind} Given a sequence  $V_1,\dotsc V_n$ of
    finite dimensional $E$-subvector spaces of $E[R]$,  there is a
    sequence $z_1,\dotsc,z_n$ of non-zero elements of $R$ such that
    the subspaces $V_1z_1,\dotsc,V_nz_n$ are independent.
\end{claim}

\begin{claimproof} Assume that $z_1,\dotsc,z_{k-1}$ have been already
  constructed. Let $z$ be as in Claim \ref{C:nonunit}. If we consider
  the sequence of ideals $z^k E[R]$, an easy case of Krull's Intersection
  Theorem (\cite[Theorem VI 7.6]{sL84}) applied to the noetherian integral
  domain $E[R]$ yields that
\[ 0=\bigcap\limits_{k\in \N} z^kE[R].\]

\noindent Choose some natural number $N_k$ large enough such that
\[(V_1z_1+\dotsb+V_{k-1}z_{k-1})\cap z^{N_k} E[R]=0,\] and set
$z_k=z^{N_k}$.
\end{claimproof}

\noindent Let us now prove that the matrix $A$ satisfies
\eqref{linpaar}. Let $V_j$ be the $E$-vector space generated by
$a_{1,j},\dotsc,a_{m,j}$, that is, by the $j$-th column of $A$. Choose
$0\neq z_j$ in $R$ as in Claim \ref{C:seq_ind}, and write each
$z_j=r_j(\bar a)$, for some polynomial $r_j(\bar x)$ with integer
coefficients. Since $A$ satisfies \eqref{infreduktion}, there is a
non-trivial tuple $\zeta$ in $E^m$ such that
  \[\sum_{i=1}^m\zeta_i\bigl(\sum_{j=1}^na_{i,j}z_j\bigr)=
  \sum_{j=1}^n\bigl(\sum_{i=1}^m\zeta_ia_{i,j}\bigr)z_j=0.\]

\noindent Observe that $\bigl(\sum_{i=1}^m\zeta_ia_{i,j}\bigr)z_j$
belongs to $V_jz_j$. The subspaces $V_iz_1,\ldots, V_nz_n$ are
independent, so each $\bigl(\sum_{i=1}^m\zeta_ia_{i,j}\bigr)z_j$ must
equal $0$. Therefore so is \[ \sum_{i=1}^m\zeta_ia_{i,j} =0,\]

\noindent as desired.
\end{proof}
\begin{question}
  Can the integer $N$ and the polynomials $r_{i,j}$ in Proposition
  \ref{P:simple} be explicitly computed?
\end{question}
\begin{theorem}\label{T:simple}
  Every tame formula is equivalent in $\acfp$ to a conjunction of simple
  linear formulae.
\end{theorem}
\begin{proof}
  By Lemma \ref{L:linear}, it suffices to show that every linear
  formula is is equivalent in $\acfp$ to a conjunction of simple
  linear formulae. This follows immediately from Proposition
  \ref{P:simple}, once we remark that the polynomial equation
  $q(x)\doteq 0$ is equivalent  in $\acfp$ to the simple linear formula
  $\dep_1(q(x))$.
\end{proof}
Together with Corollary \ref{C:Tame_BK}, we deduce another proof of
\cite[Theorem 1.1]{aG16}, valid in all characteristics:

\begin{cor}\label{C:simple}
  In the theory $\acfp$ of proper pairs of algebraically closed field,
  every formula is equivalent in to a boolean combination of simple
  tame formulae.
\end{cor}

\noindent In particular, we obtain another proof of the
equationality of $\acfp$ in characteristic $0$, for every simple
linear formula is an equation in a differential field: Indeed, the
formula $\dep_s(x_1,\dotsc,x_s)$ is equivalent to the differential
equation $\W(x_1,\dotsc,x_s)\doteq0$.

Corollary \ref{C:simple} implies, together with Corollary
\ref{C:conj_tame}, that a finite conjunction of linear formulae is
again linear. However, we do not think that the same holds for simple
linear formulae.

\noindent A key point in the proof of \cite[Theorem 1.1]{aG16} is the
fact that each $\LLD$-function $\lambda_n^i$ defines, on its domain, a
continuous function with respect to the topology generated by
instances of simple linear formulae \cite[Proposition 2.6]{aG16}. We
will conclude with an easy proof that all functions
$\lambda_n^i\times\id\times\dotsb\times\id$ are continuous with
respect to this topology. For this, we need an auxiliary definition
(cf. Definition \ref{D:lambda_tame}):

\begin{definition}\label{D:lambda_formula}
  The collection of $\lambda_P$-formulae is the smallest collection
  of formulae in the language $\LLD$, closed under
  conjunctions and containing all polynomial equations, such that, for
  any natural number $n$ and polynomials $q_0, \ldots, q_n$ in
  $\Z[x]$, given a $\lambda_P$-formula $\psi(x, z_1,\ldots, z_n)$,
  the formula
  \begin{multline*}
    \phi(x)=\dep_n(q_1(x),\ldots, q_n(x))\;\;\lor\;\;\\
    \bigl(\,\ldef{q_0(x),\ldots, q_n(x)}\;\land\;\, \psi(x,
    \overline{\lambda}_n(q_0(x), \ldots, q_n(x)))\,\bigr)
  \end{multline*}
  is $\lambda_P$-tame, where $\ldef{y_0,\dotsc, y_n}$ is an
    abbreviation for
    \[\neg\dep_n(y_1,\dotsc,y_n)\land\dep_{n+1}(y_0,\dotsc,y_n).\]
\end{definition}

\begin{prop}\label{P:tame_lambda}
Up to equivalence in $\acfp$,  tame formulae  and
  $\lambda_P$-formulae coincide.
\end{prop}
\begin{proof}

Notice that every simple linear formula is $\lambda_P$-tame, since
\[\dep_n(y_1,\dotsc,y_n)\;\;\Leftrightarrow\;\;
  \dep_n(y_1,\dotsc,y_n)\lor\left( \ldef{0,y_1,\ldots,y_n}\, \land \,
    (1\doteq 0)\right).\]

\noindent By Theorem \ref{T:simple}, we conclude that all tame
formulae are $\lambda_P$-tame.

\noindent We prove the other inclusion by induction on the degree of the
$\lambda_P$-formula $\phi(x)$. Polynomial equations are clearly
tame. By Corollary \ref{C:conj_tame}, the conjunction of tame formulae
is again tame. Thus, we need only show that $\phi(x)$ is tame,
whenever
\[ \phi(x)=\dep_n(q_1,\ldots, q_n)\;\lor\;\ \bigl(\,\ldef{q_0,\ldots,
    q_n}\;\land\;\, \psi(x, \overline{\lambda}_n(q_0, \ldots,
  q_n))\bigr), \]

\noindent for some tame formula $\psi(x, z_1,\ldots, z_n)$.   Write

\[ \psi(x,z)\;\;=\;\;\exists \zeta \in
  P^s\Bigl(\neg\zeta\doteq 0\;\land\;\Land_{k\leq m}
  p_k(x,z,\zeta)\doteq 0 \Bigr),\]
\noindent for some
polynomials $p_1(x,z,u),\ldots, p_m(x,z,u)$ with integer coefficients
and homogeneous in $u$.

\noindent Homogenising with respect to the variables $z_0, z_1,\ldots,
z_n$, there is some
natural number $N$ such that, for each $k\leq m$,

\[p_k(x,z_0\inv z,u)z_0^N=r_k(x,z_0,z,u),\]

\noindent where $r_k$ is homogeneous in $(z_0,z)$ and in $u$,
separately. Thus,

  \begin{multline*}
    \acfp\models \biggl( \phi(x) \longleftrightarrow
    \biggl(\exists(\zeta_0,\zeta)\in P^{n+1}\;\exists \upsilon \in
    P^s\Bigl(\neg(\zeta_0,\zeta) \doteq0\land\neg\upsilon\doteq 0\; \\
    \land \zeta_0 q_0(x)+\cdots+\zeta_nq_n(x)\doteq0\land \Land_{k\leq
      m} r_k(x,\zeta_0, \zeta,\upsilon)\doteq 0 \Bigr)\biggr) \biggr).
  \end{multline*}

  \noindent The right-hand expression is a tame formula, by Lemma
  \ref{L:segre}, and so is $\phi$, as desired.
\end{proof}



\end{document}